\documentclass[a4paper,10pt,reqno]{amsart}
\usepackage[foot]{amsaddr}
\usepackage[english]{babel}
\usepackage[utf8]{inputenc}
\usepackage{graphicx}
\graphicspath{{figures/}}
\usepackage[colorinlistoftodos]{todonotes} 
\usepackage{subcaption}
\usepackage[colorlinks=true,linkcolor=black,urlcolor=black,citecolor=black]{hyperref}
\usepackage{hyperref}
\usepackage{amsmath}
\usepackage{amsfonts}
\usepackage{amssymb,amsmath}
\usepackage{amsthm}
\usepackage{fixmath}
\usepackage{enumerate}
\usepackage[shortlabels]{enumitem}
\usepackage{mathtools}
\usepackage{xcolor}
\numberwithin{equation}{section}
\numberwithin{figure}{section}

\newtheorem{theorem}{Theorem}[section]
\newtheorem{lemma}[theorem]{Lemma}
\newtheorem{assumption}[theorem]{Assumption}
\newtheorem{proposition}[theorem]{Proposition}

\newtheorem{remark}{Remark}[section]

\newcommand{\x}{{{x}}}
\newcommand{\xx}{{y}}
\newcommand{\n}{{n}}
\renewcommand{\d}{\mathop{}\!\mathrm{d}}
\newcommand{\D}{\Omega}
\newcommand{\DC}{\widetilde{\D}}
\newcommand{\DI}{\D_{\delta}}
\newcommand{\DD}{\D\cup\DI}
\newcommand{\R}{\mathbb{R}}
\newcommand{\N}{\mathbb{N}}
\newcommand{\kernel}{\gamma}

\newcommand{\Vd}{V_{\delta}}

\newcommand{\q}{q} 
\newcommand{\hq}{h} 
\newcommand{\hu}{hu} 
\newcommand{\hs}{h_s} 
\newcommand{\hd}{h_{\delta}} 

\newcommand{\mcR}{\mathcal{R}}

\newcommand{\mcX}{\mathcal{X}}

\newcommand{\mcL}{\mathcal{L}}

\newcommand{\mbR}{\mathbb{R}}
\newcommand{\mbRn}{{\mathbb{R}^n}}

\newcommand{\gam}{{\gamma}}

\newcommand{\ud}{{u_d}}
\newcommand{\Q}{{Q}}
\newcommand*{\HD}[1]{{H_{\D}^{#1}(\R^n)}}
\renewcommand{\S}{S}
\newcommand*{\norms}[1]{{\left|{#1}\right|}}
\newcommand*{\norm}[1]{{\left\lVert#1\right\rVert}}
\newcommand*{\norml}[1]{{\left\lVert#1\right\rVert}_{L^2(\D)}}

\newcommand{\abs}[1]{\left|#1\right|}

\begin{document}
\title[]{An optimization-based approach to parameter learning for fractional type nonlocal models}
\author{Olena Burkovska${}^{1}$ \and Christian Glusa${}^{2}$ \and Marta D'Elia${}^{3}$}
\address{${}^1$ Computer Science and Mathematics Division, Oak Ridge National Laboratory, USA}
\email{burkovskao@ornl.gov}

\address{${}^2$ Center for Computing Research, Sandia National Laboratories, NM, USA}
\email{caglusa@sandia.gov}
\address{${}^3$ Computational Science and Analysis, Sandia National Laboratories, CA, USA}
\email{mdelia@sandia.gov}

\subjclass[]{}

\thanks{
This manuscript has been authored in part by UT-Battelle, LLC, under contract DE-AC05-00OR22725 with the US Department of Energy (DOE) and 
supported by the U.S. Department of Energy, Office of Advanced Scientific Computing Research, Applied Mathematics Program under the award numbers ERKJ345 and ERKJE45. The US government retains and the publisher, by accepting the article for publication, acknowledges that the US government retains a nonexclusive, paid-up, irrevocable, worldwide license to publish or reproduce the published form of this manuscript, or allow others to do so, for US government purposes. DOE will provide public access to these results of federally sponsored research in accordance with the DOE Public Access Plan (\texttt{http://energy.gov/downloads/doe-public-access-plan}).\\
This work was partially supported by the Sandia National Laboratories (SNL) Laboratory-directed Research and Development program and by the U.S. Department of Energy, Office of Advanced Scientific Computing Research under the Collaboratory on Mathematics and Physics-Informed Learning Machines for Multiscale and Multiphysics Problems (PhILMs) project. SNL is a multimission laboratory managed and operated by National Technology and Engineering Solutions of Sandia, LLC., a wholly owned subsidiary of Honeywell International, Inc., for the U.S. Department of Energy's National Nuclear Security Administration under contract DE-NA-0003525. This paper, SAND2020-10989, describes objective technical results and analysis. Any subjective views or opinions that might be expressed in this paper do not necessarily represent the views of the U.S. Department of Energy or the United States Government.
}

\begin{abstract}
Nonlocal operators of fractional type are a popular modeling choice for applications that do not adhere to classical diffusive behavior; however, one major challenge in nonlocal simulations is the selection of model parameters. In this work we propose an optimization-based approach to parameter identification for fractional models with an optional truncation radius. 
We formulate the inference problem as an optimal control problem where the objective is to minimize the discrepancy between observed data and an approximate solution of the model, and the control variables are the fractional order and the truncation length.
For the numerical solution of the minimization problem we propose a gradient-based approach, where we enhance the numerical performance by an approximation of the bilinear form of the state equation and its derivative with respect to the fractional order. 
Several numerical tests in one and two dimensions illustrate the theoretical results and show the robustness and applicability of our method.
\end{abstract}

\maketitle
\section{Introduction}\label{sec:intro}
Nonlocal models provide an improved predictive capability for several scientific and engineering applications thanks to their ability to describe effects that classical partial differential equations (PDEs) fail to capture. These effects include multiscale and anomalous behavior such as super- and sub-diffusion in, e.g., subsurface modeling \cite{Benson2000,Meerschaert2006,Schumer2001,Schumer2003} or sharp interface dynamics in phase field models~\cite{Burkovska2020CH,DuYang2016,DuJuLiQiao2018}. Furthermore, applications in material science, corrosion, turbulent flow, and geoscience exhibit hierarchical features that do not adhere to classical diffusive behavior. Fractional order models are one of the most common instances of the nonlocal models. Through fractional exponent derivatives, as opposed to integer ones in the case of PDEs, fractional order operators can accurately represent multiscale complex phenomena by incorporating long range interactions in the model itself in the form of an integral operator acting over the whole space. 

Given a scalar quantity $u:\mbRn\to\mbR$, the simplest form of a fractional operator is the fractional Laplacian; in its integral form, its action on $u$ is defined as \cite{Gorenflo1997}
\begin{equation}\label{eq:fractional-laplacian}
(-\Delta)^s u(x) = C_{n,s}\int_\mbRn \frac{u(x)-u(y)}{|x-y|^{n+2s}} dy,
\end{equation}
where $s\in(0,1)$ is the {\it fractional order} and $C_{n,s}$ is a constant that depends on the dimension $n$ and $s$:
\begin{equation}\label{eq:FL_constant}
C_{n,s}=\frac{2^{2s}s\Gamma(s+n/2)}{\pi^{n/2}\Gamma(1-s)}.
\end{equation}
Here, and in what follows, the integral in~\eqref{eq:fractional-laplacian} should be understood in the Cauchy principal value sense for $s\in [1/2, 1)$. It follows from \eqref{eq:fractional-laplacian} that in fractional modeling the state of a system at a point depends on the value of the state at any other point in the space; in other words, fractional models are nonlocal. Specifically, fractional operators are special instancies of more general nonlocal operators \cite{Defterli2015,DElia2020,DElia2013,Du2012,Pang2020} of the following form:
\begin{equation}\label{eq:nonlocal-laplacian}
-\mcL u(x) = 2\int_{B_\delta(x)} (u(x)-u(y)) \gam(x,y) dy.
\end{equation}
Here, interactions are limited to a Euclidean ball $B_\delta(x)$ of radius $\delta$, often referred to as horizon or interaction radius. The symmetric nonnegative kernel\footnote{For a discussion regarding nonpositive kernels we refer the reader to \cite{Mengesha2013} and for non-symmetric kernels to \cite{DElia2017,Felsinger2015}.} $\gam(x,y)$ is a modeling choice and determines regularity properties of the solution. First, for $\delta=\infty$ and for the fractional type kernel $\gam(x,y)=(C_{n,s}/2)|x-y|^{-n-2s}$ the nonlocal operator in \eqref{eq:nonlocal-laplacian} is equivalent to the fractional Laplacian in \eqref{eq:fractional-laplacian}; second, it has been shown \cite{DElia2013} that for that choice of $\delta$ and $\gam$ solutions corresponding to the nonlocal operator \eqref{eq:nonlocal-laplacian} converge to the ones corresponding to the fractional operator \eqref{eq:fractional-laplacian} as $\delta\to\infty$. In this paper we consider nonlocal operators with finite-range interactions and characterized by fractional type kernels and we refer to them as {\it truncated fractional operators} (see \cite{DElia2020} for a detailed classification of these operators and relationships between them). 

\smallskip
The highly descriptive power of these operators comes at the price of several modeling and computational challenges, including the unresolved treatment of nonlocal interfaces \cite{Alali2015,Capodaglio2020}; the prescription of nonlocal analogues of boundary conditions \cite{Cortazar2008,DEliaNeumann2019,Lischke2020}; the increasing computational cost as the extent of the nonlocal interactions increases \cite{AinsworthGlusa2017,AinsworthGlusa2018,DElia-ACTA-2020,DEliaFEM2020,Pasetto2019,Silling2005meshfree,Wang2010}, i.e. as $\delta\to\infty$; and the uncertainty and sparsity of model parameters and data \cite{Antil2019FracControl,DElia2019FracControl,DElia2014DistControl,DElia2016ParamControl,DElia2019Imaging,Glusa2019,Gulian2019,Pang2019fPINNs,Pang2017discovery,You2020Regression}. In this work we focus on the latter challenge and design an optimization-based data-driven parameter identification strategy.

Parameter identification problems for nonlocal and fractional operators have been addressed in the literature both with an optimal control approach and a machine learning approach. The literature on uncertainty quantification in the context of parameter estimation is still limited, \cite{Pang2017discovery} being perhaps the only relevant paper. In several works the identification problem is formulated as an optimal control problem where the objective functional is the mismatch between the solution and a given target and the constraints are the nonlocal equations. As an example, in \cite{DElia2016ParamControl} the authors learn a variable diffusion coefficient characterizing the kernel $\gam$ for a class of nonlocal operators that includes truncated fractional operators. An approach similar to the one presented in this work for the estimation of the fractional order of the {\it spectral} fractional Laplacian \cite{Musina2014} has been presented in \cite{Sprekels2017} and further analyzed in \cite{Antil2018sSpectralControl}. Based on these works, applications to image processing have also been considered \cite{Bartels2020}. On the other hand, several other works tackled the identification problem with a machine learning approach; as an example, in \cite{Pang2020,Pang2019fPINNs} the authors employ physics-informed neural networks to describe the nonlocal solution and learn model parameters such as $s$ and $\delta$ through minimization of a loss function given by the solution mismatch and the residual of the state equation. A similar approach, referred to as {\it operator regression} and also based on minimization of the equation residual, is utilized in \cite{You2020Regression} where the authors learn nonlocal kernels, including fractional type kernels, in a least-squares regression setting. Due to the complexity of the functionals involved in their formulations, the machine learning approaches mentioned above do not provide any thoeretical analysis, which is, instead, a central aspect of our work.

Our paper is aligned with the former set of approaches and formulates the learning strategy as an optimal control problem where the controls are the model parameters $s$ and $\delta$, the state equation is a Poisson problem for the truncated fractional Laplacian and the cost functional is a matching functional, i.e. a measure of the difference between the nonlocal solution and a given target. To the best of our knowledge, this is the first work that analyzes the optimal control problem with respect to the fractional order and the truncation legnth of the {\it integral} truncated fractional Laplacian. We also point out that our analysis holds for the case of infinite truncation length, for which the truncated nonlocal operator coincides with the fractional Laplacian itself; this was also missing in the literature. Additionally, we provide and analyze efficient approximations of the bilinear form induced by the variational formulation of the problem and of its derivatives, with the goal of improving the computational performance of the optimization algorithm. These approximations not only are beneficial in our optimal control setting, but have the potential to improve the performance of any algorithm that involves the evaluation of bilinear forms and corresponding derivatives when parameters, such as $s$ and $\delta$, vary throught the algorithm.
More specifically, our main contributions are:\\
1) the rigorous analysis of regularity properties of the bilinear form of the state equations and of the control-to-state map;\\
2) the formulation of a fractional-equation-constrained optimization problem and the derivation of optimality conditions;\\
3) the rigorous analysis, including error estimates, of efficient approximations of the bilinear form and its derivatives with respect to the fractional order (allowing for affine decompositions) and the truncation length;\\
4) the design of a gradient-based algorithm for the solution of the optimization problem, including the construction of an approximate gradient (based on 3) and its error analysis, that allows for fast computations and a complexity study for the resulting discretized formulation.

\subsection{Outline of the paper}
The paper is organized as follows.
In Section \ref{sec:preliminaries} we introduce rigorous notation and recall relevant theoretical results.
In Section \ref{sec:regularity} we analyze regularity properties of the bilinear form of the state equations with respect to the model parameters $\delta$ and $s$ and of the control-to-state map.
In Section \ref{sec:optimization} we introduce the optimal control problem, analyze the existence of solutions and propose an adjoint-based approach for its solution.
In Section \ref{sec:interp-form} we introduce approximations of the bilinear form and its {derivatives and} analyze their accuracy.
In Section~\ref{sec:discretization-error-estimates} we discuss the discretization of state and adjoint problems.
We analyze the convergence subject to discretization and interpolation error of state, adjoint and gradient of the cost functional.
In Section \ref{sec:numerics} we illustrate our theoretical results with several one- and two-dimensional computational tests.
Finally, in Section \ref{sec:conclusion} we report concluding remarks and provide insights on future research directions.

\section{Preliminaries}\label{sec:preliminaries}
Let $\D\subset\R^n$, $n\ge1$ be a bounded domain with $C^{\infty}$-boundary $\partial\D$. 
We introduce a parameter $\q=(s,\delta)\in\Q$, where 
\[
\Q:=\{\q\in\R^2\colon 0<s<1, \;\;0<\delta<\infty\},
\]
and we consider the parametrized nonlocal operator
\begin{equation}\label{nonl_operator}
-\mathcal{L}(\q)u(x):={2} \int_{\R^n}(u(\x)-u(\xx))\kernel(\x,\xx;\q)\d\xx,\quad\x\in\R^n,
\end{equation}
with a fractional type kernel function defined as
\begin{equation}
\kernel(\x,\xx;q):= \frac{\mcX_{B_\delta(\x)}(\xx)}{|\x-\xx|^{n+2s}}.
\label{eq:kernel_delta}
\end{equation}
Here, $\mcX_{B_\delta(\x)}$ denotes the characteristic function over the Euclidean ball $B_{\delta}(\x):=\{y\in\R^n\colon |\x-\xx|\leq\delta\}$, the parameter $\delta$ defines the extent of the nonlocal interactions, and $s$ represents the fractional order of the truncated fractional operator. 
We also introduce a vector $\hq:=[\hs,\hd]^T\in\Q$, and by a slight abuse of notation we denote by $\abs{\hq}^2:=\norm{\hq}^2=\abs{\hs}^2+\abs{\hd}^2$ the usual Euclidean norm.

\begin{remark}
For convenience of notation we omit the scaling constant $C_{n,s}$ in~\eqref{eq:kernel_delta}, and note that the approaches {developed in this work} can be directly extended to the case of $C_{n,s}$ defined in~\eqref{eq:FL_constant}.
\end{remark}

To capture nonlocal interactions outside of $\D$, we introduce the so-called interaction domain {$\DI:=\{\xx\in\R^n\colon \; \exists \;\x\in\D : |\x-\xx|<\delta\}$}, and in terms of this notation the operator in~\eqref{nonl_operator} can also be expressed as 
\begin{multline*}
-\mathcal{L}(\q)u(x)= 2\int_{\DD}(u(\x)-u(\xx))\kernel(\x,\xx;\q)\d\xx=2\int_{(\DD)\cap B_{\delta}(\x)}\frac{u(\x)-u(\xx)}{|\x-\xx|^{n+2s}}\d\xx.
\end{multline*}
We consider the following nonlocal problem with homogeneous Dirichlet-type volume constraints: 
\begin{equation}
\begin{aligned}
-\mathcal{L}(\q)u(x)&=f(\x),\quad&&\x\in\D,\\
u(\x)&=0,\quad&&\x\in\DI.\label{eq:nonl_strong}
\end{aligned}
\end{equation}
For convenience of notation we suppose that $u$ is extended by zero outside of $\DD$, and, by a slight abuse of notation, we still denote it by $u$. When necessary, we chose the whole $\R^n$ as a common spatial domain.
In what follows we often use the relation $a\lesssim b$, implying that $a\leq C b$, where $C>0$ is some constant that may be different in every case.

\subsection{Functional spaces}
We denote by $L^2(\DC)$ the Lebesgue space of square-integrable functions and by $H^1(\DC)$ and $H^1_0(\DC)$ the usual Sobolev spaces, where $\DC$ is given either by $\D$, $\DD$ or $\R^n$.
For $u,v\in L^2(\R^n)$ we define the bilinear form 
\begin{align}
a(u,v):=\int_{S_\delta} \frac{\left(u(\x)-u(\xx)\right)\left(v(\x)-v(\xx)\right)}{|\x-\xx|^{n+2s}}\d\xx\d x,
\label{bil_form}
\end{align}
where $S_\delta$ is the strip
\begin{equation}
S_{\delta}=\{(\x,\xx)\in\R^{2n}\colon 
|\x-\xx|\leq\delta\},
\end{equation}
and for $s\in (0,1)$, $0<\delta\leq\infty$ we introduce the following nonlocal function spaces
\[
\Vd = \{v\in L^2(\DD)\colon a(v,v)<\infty, \;\; v=0\;\;\text{a.e. on}\;\;\DI\},
\]
equipped with the inner product $(u,v)_{\Vd}:=a(u,v)$ and norm $\norm{u}^2_{\Vd}=(u,u)_{\Vd}$.
Next, we define the fractional Sobolev spaces of order $s\in(0,1)$
\begin{equation*}
 H^s(\D):=\big\{v\in L^2(\D)\colon |v|_{H^s(\D)}<\infty\big\},
\end{equation*}
where the Gagliardo seminorm is defined as
\begin{equation*} 
|v|_{H^s(\D)}^2:=\int_{\D}\int_{\D}\frac{|v(\x)-v(\xx)|^2}{|\x-\xx|^{n+2s}}
\d\xx\d\x.
\end{equation*}
For $s>1$ not an integer, we define $H^s(\D)$, $s=m+\sigma$ with $m\in\mathbb{N}$ and $\sigma \in (0,1)$, as
\begin{equation*}
 H^s(\D):=\{v\in H^m(\D)\colon {D^\alpha v} \in {H^\sigma(\D)} \text{ for } 
|\alpha|=m\},
\end{equation*}
together with the semi-norm $\norms{v}^{2}_{H^s(\D)}=\norms{v}^{2}_{H^m(\D)}+\sum_{|\alpha|=m}\norms{D^\alpha v}^{2}_{
H^\sigma(\D) }$.
The space $H^s(\D)$ is a Hilbert space, endowed with the norm 
$ \norm{v}^{2}_{H^s(\D)}=\norm{v}^{2}_{L^2(\D)}+\norms{v}^{2}_{H^s(\D)}$.
Additionally, we define the space incorporating the volume constraints given by
\begin{equation*}
 \HD{s}:=\{v\in H^s(\R^n)\colon v=0 \text{ on }\ \R^n\setminus\D\},
\end{equation*}
that is endowed with the semi-norm of $H^s(\R^n)$, {and by $(\cdot,\cdot)_{\HD{s}}$ we denote the corresponding inner product}. For negative exponents, we define the corresponding spaces by duality $H^{-s}(\Omega) = (H_{\D}^s(\R^n))^{\prime}$. 

We point out that for $0<\delta\leq\infty$ the nonlocal spaces $\Vd$ are all equivalent, see~\cite[Proposition~2.2]{BG2019rbm}. Furthermore, for $\Vd$ is equivalent to $\HD{s}$, which implies that we can equivalently work either with $\Vd$ or $\HD{s}$. 

We also recall that for all $v\in\HD{s}$ and $s\in(0,1)$ the following nonlocal Poincar\'e inequality holds
\begin{equation}
\norm{v}_{L^2(\D)}\leq C_P\norm{v}_{\HD{s}},\quad C_P>0.\label{eq:poincare}
\end{equation}

\subsection{Variational formulation}
We consider the following weak formulation of problem~\eqref{eq:nonl_strong}. For $\q\in\Q$, {$f\in H^{r}(\D)$, $r\geq -s$}, $\varepsilon>0$, find $u(\q)\in\HD{s}$, such that
\begin{equation}
a(u,v;\q)=(f,v)_{L^2(\D)},\quad\forall v\in\HD{s}.\label{eq:var_form}
\end{equation}
The existence and uniqueness of the solution of~\eqref{eq:var_form} follows by  the Lax-Milgram lemma. Moreover, the following improved regularity result holds true~\cite{BG2019vi,grubb2015}.
\begin{theorem}\label{thm:regularity_linear}
Let $\D$ be a domain with $C^{\infty}$ boundary $\partial\D$, and let $f\in  H^r(\D)$, $r\geq -s$, and let $u\in\HD{s}$ be the solution of~\eqref{eq:var_form} with the kernel~\eqref{eq:kernel_delta} for $\delta>0$. Then, for any $\varepsilon>0$ the following holds
\begin{equation}
\norm{u}_{H_\D^{s+\alpha}(\R^n)}\lesssim\norm{f}_{{H^r(\D)}},\label{nonl_regularity}
\end{equation} 
where $\alpha=\min\{s+r,1/2-\varepsilon\}$. 
\end{theorem}
\begin{assumption}\label{ass:f-regularity}
We let the forcing term $f$ be such that $f\in H^{1/2-\varepsilon}(\D)$.
\end{assumption}
Note that Assumption \ref{ass:f-regularity} implies that $u\in\HD{s+1/2-\varepsilon}$, $\varepsilon>0$; this is the highest regularity for the solution $u$.

To conclude this section, we introduce a solution operator $\S:\Q\to\HD{s}\subset L^2(\D)$, such that $u(\q)=\S(\q)$ is a unique solution of~\eqref{eq:var_form}, i.e., $\S(\q)$ is such that
\begin{equation}
a(\S(\q),v;\q)=(f,v)_{L^2(\D)},\quad\forall v\in\HD{s}.\label{eq:var_form_s}
\end{equation}

\section{Regularity of the bilinear form and of the Control-to-State map}\label{sec:regularity}
In this section we recall from~\cite{BG2019rbm} the parametric regularity results with respect to $s$ and $\delta$ for the bilinear form associated with the state equation and then analyze properties of the control-to-state map.
\begin{proposition}[$\delta$-regularity]\label{prop:delta_reg}
For $\delta\in(0,\infty)$, $s\in(0,1)$, and $\kernel$ defined in~\eqref{eq:kernel_delta}, the following hold true for all $u,v\in\HD{s}$:
\begin{enumerate}
\item[(i)] the bilinear form  $a(\cdot,\cdot;\delta)$ is differentiable with respect to $\delta$ with
\begin{equation}
\label{eq:der_a}
a^{\prime}_\delta(u,v;\delta)
= \int_{\R^n} \int_{\partial B_\delta(\x)} \frac{\left(u(\x)-u(\xx)\right)\left(v(\x)-v(\xx)\right)}{|\x-\xx|^{n+2s}}\d\xx \d\x,
\end{equation}
where $\partial B_\delta(\x)$ is the surface of the ball of radius $\delta$ centered at $\x$, i.e.\ $\partial B_\delta(\x) = \{\xx \in \R^n\colon |\x-\xx|= \delta\}$, and the inner integral is understood as a surface integral;
\item[(ii)] the derivative  $a^\prime_{\delta}(\cdot,\cdot;\delta)$ is bounded
\begin{equation}
\norms{a^{\prime}_\delta(u,v;\delta)}\lesssim\norm{u}_{L^2(\D)}\norm{v}_{L^2(\D)};
\label{eq:bdd_der_a_delta}
\end{equation}
\item[(iii)] the derivative  $a^\prime_{\delta}(\cdot,\cdot;\delta)$ is H\"older continuous, i.e., for $\theta\in[0,1]$ and $\delta_1,\delta_2\in(0,\infty)$ it holds
\begin{equation}
\norms{a^{\prime}_{\delta}(u,v;\delta_1)-a^{\prime}_{\delta}(u,v;\delta_2)}\lesssim \norms{\delta_1-\delta_2}^\theta\norm{u}_{\HD{\theta}}\norm{v}_{L^2(\D)},\label{eq:hodler_der_a_delta}
\end{equation}
where $\HD{0}\equiv L^2(\D)$ and $\HD{1}\equiv H^1_0(\D)$.
\end{enumerate}
\end{proposition}
\begin{proof}
The first two statements follow directly from~\cite[Theorem~3.3]{BG2019rbm}. For $\theta=1$ and $\theta=0$, assertion (iii) has been proven in~\cite[Theorem~3.3]{BG2019rbm}. Then, applying the interpolation theorem we obtain the remaining estimate for $\theta\in(0,1)$ and conclude the proof. 
\end{proof}

\begin{proposition}[$s$-regularity]\label{prop:s_reg}
Let $\delta\in(0,\infty)$ and $s\in (0,1)$.
Then, for $u\in\HD{s+\alpha}$, $0<\alpha<1-s$, $v\in\HD{s}$, and $\kernel$ defined in~\eqref{eq:kernel_delta}, the following hold true:
\begin{enumerate}
\item[(i)] the bilinear form $a(u,v;s)$ is infinitely many times differentiable, i.e., for $k=1,2,\dots$, it holds
\begin{align}
a_s^{(k)}(u,v;s):=(-2)^k\int_{S_{\delta}}\frac{\left(u(\x)-u(\xx)\right)\left(v(\x)-v(\xx)\right)\log^k(|\x-\xx|)}{|\x-\xx|^{n+2s}}\d\xx \d\x;
\end{align}
\item[(ii)] for $u\in\HD{s_2}$, $v\in\HD{s_1}$, $s_1,s_2\in(0,1)$, $s\in(0,(s_1+s_2)/2)\subset(0,1)$, the derivative $a_s^{(k)}(u,v;s)$ is bounded
\begin{equation}
|a_s^{(k)}(u,v;s)|\leq C(k,\hat{\varepsilon}) \norm{u}_{\HD{s_2}}\norm{v}_{\HD{s_1}},\label{eq:bdd_der_a_s}
\end{equation}
where $C(k,\hat{\varepsilon})=2^k\left((k/(e\hat{\varepsilon}))^k+\delta^{\hat{\varepsilon}}(\log(\delta))_{+}^k\right)$, $\hat{\varepsilon}=s_1+s_2-2s>0$;
\item[(iii)] the derivative $a_s^{(k)}(u,v;s)$ is locally Lipschitz continuous:
\begin{equation}
|a_s^{(k)}(u,v;s)-a_s^{(k)}(u,v;s+\eta)|\lesssim \eta\norm{u}_{\HD{s+\alpha}}\norm{v}_{\HD{s}},\label{lipschitz_der_a}
\end{equation}
where $0<\alpha<1-s$, $0<\eta<\alpha/2$ and the hidden constant depends on $\D$, and continuously on $s$ and $\alpha$.
\end{enumerate}
\end{proposition}
\begin{proof}
Since, the first two statements follow directly from~\cite[Lemma~5.1]{BG2019rbm}, it remains to prove the validity of~\eqref{lipschitz_der_a}. Using the results of~\cite[Lemma~5.1]{BG2019rbm} with $s_2=s+\alpha$ and $s_1=s$, we obtain 
\begin{equation*}
\begin{aligned}
   |a_s^{(k)}(u,v;s)-a_s^{(k)}(u,v;s+\eta)|
& \leq \eta\max_{\xi\in[s,s+\eta]}|a_s^{(k+1)}(u,v;\xi)|\\
& \lesssim\eta\norm{u}_{\HD{s+\alpha}}\norm{v}_{\HD{s}},
\end{aligned}
\end{equation*}
which concludes the proof.
\end{proof}
Combining the results of Proposition~\ref{prop:delta_reg} and Proposition~\ref{prop:s_reg} we derive the following estimates for the mixed derivatives.
\begin{proposition}[Mixed derivatives]
For $\delta\in(0,\infty)$, $s\in(0,1)$, and $\kernel$ defined in~\eqref{eq:kernel_delta}, the following hold true for all $u\in\HD{s+\alpha}$, $0<\alpha<1-s$, $v\in\HD{s}$:
\begin{enumerate}
\item[(i)] there exists a second mixed derivative of  $a(\cdot,\cdot;\delta)$ with respect to $\delta$ and $s$: 
\begin{equation}
\begin{aligned}
a^{\prime\prime}_{s,\delta}(u,v;s,\delta):=&\ \frac{\partial^2}{\partial s\partial\delta}\left(a(u,v;s,\delta)\right)\\
= &\ -2\int_{\R^n} \int_{\partial B_\delta(\x)} \frac{\left(u(\x)-u(\xx)\right)\left(v(\x)-v(\xx))\log(|\x-\xx|\right)}{|\x-\xx|^{n+2s}}\d\xx \d\x,
\end{aligned}\label{eq:der_a_deltas}
\end{equation}
where $\partial B_\delta(\x)$ is defined as in Proposition~\ref{prop:delta_reg}.
\item[(ii)] the mixed derivative  $a^{\prime\prime}_{s,\delta}(\cdot,\cdot;s,\delta)$ is bounded
\begin{equation}
\norms{a^{\prime\prime}_{s,\delta}(u,v;s,\delta)}\lesssim\norm{u}_{L^2(\D)}\norm{v}_{L^2(\D)},
\label{eq:bdd_der_a_deltas}
\end{equation}
with the hidden constant depending continuously on $s$ and $\delta$.
\end{enumerate}
\end{proposition}

Next, we analyze the properties of the map $\S:\Q\to L^2(\D)$. The following proposition shows that $\S$ is Lipschitz continuous.

\begin{proposition}[Lipschitz continuity of $\S(\q)$]\label{prop:cont_S}
Let $\delta\in(0,\infty)$ and $s\in (0,1)$.
Then, for any $\q\in\Q$, $S(\q)\in\HD{s+1/2-\varepsilon}$. Moreover, for $0<\alpha<\min\{(1-s)/2-\varepsilon,s\}$, $\norms{\hq}<\alpha/2$, and $\kernel$ defined in~\eqref{eq:kernel_delta}, $S(\q)$ is locally Lipschitz continuous:
\begin{equation}
\norm{\S(\q)-\S(\q+\hq)}_{\HD{s+\alpha}}\lesssim\norms{\hq}\norm{\S(\q+\hq)}_{\HD{s+2\alpha}}.\label{lipschitz_S}
\end{equation}
\end{proposition}
\begin{proof}
We note that the first statement $\S(\q)\in\HD{s+1/2-\varepsilon}$ follows directly from Theorem~\ref{thm:regularity_linear}. Then, for $\q,\hq\in\Q$ using~\eqref{eq:var_form_s} we can express
\begin{multline*}
a(\S(\q)-\S(\q+\hq),v;\q)=a(\S(\q),v;\q)-a(\S(\q+\hq),v;\q)\pm a(\S(\q+\hq)v;\q+\hq)\\
=a(\S(\q+\hq),v;\q+\hq)-a(\S(\q+\hq),v;\q).
\end{multline*}
Let $\widetilde{\q}(\tau)=\q+\tau\hq$, $\tau\in(0,1)$, and $g(\tau):=a(\S(\widetilde{\q}(1)),v;\widetilde{\q}(\tau))$, then we obtain that $g^{\prime}(\tau)=a^\prime_\q(\S(\q+\hq),v;\q+\tau\hq)(\hq)$ and the following holds
\begin{multline*}
a(\S(q+\hq),v;\q+\hq)-a(\S(\q+\hq),v;\q)=g(1)-g(0)=\int_{0}^1g^\prime(\xi)\d\xi\\
=\int_0^1 a^\prime_\q(\S(\q+\hq),v;\q+\xi\hq)(\hq)\d\xi.
\end{multline*}
Using~\eqref{lipschitz_der_a} with $2\alpha-1/2+\varepsilon\leq\hs<\alpha/2$, where $0<\alpha<\min\{(1-s)/2-\varepsilon,s\}$, $\varepsilon>0$, and~\eqref{eq:bdd_der_a_delta} with~\eqref{eq:poincare} we obtain 
\begin{multline*}
\abs{a^{\prime}_\q(\S(\q+\hq),v;\q+\xi\hq)(\hq)}\\
\leq\abs{\hs}\abs{a^{\prime}_s(\S(\q+\hq),v;\q+\xi\hq)}+\abs{\hd}\abs{a^{\prime}_{\delta}(\S(\q+\hq),v;\q+\xi\hq)}\\
\lesssim\abs{\hs}\norm{\S(\q+\hq)}_{\HD{s+2\alpha}}\norm{v}_{\HD{s-\alpha}}+\abs{\hd}\norm{\S(\q+\hq)}_{L^2(\D)}\norm{v}_{L^2(\D)}\\
\lesssim\norms{h}\norm{\S(\q+\hq)}_{\HD{s+2\alpha}}\norm{v}_{\HD{s-\alpha}}.
\end{multline*}
We note that under the above conditions on $\hs$ and $\alpha$, we have that $\S(\q+\hq)\in\HD{s+\hs+1/2-\varepsilon}\subset\HD{s+2\alpha}$, and hence $\norm{u}_{\HD{s+2\alpha}}$ is well-defined. Furthermore, if we introduce $F(v):=\int_0^1 a^\prime_\q(\S(\q+\hq),v;\q+\xi\hq)(\hq)\d\xi$, then the above also implies that $F\in(\HD{s-\alpha})^*= H^{-s+\alpha}(\D)$. Now, applying Theorem~\ref{thm:regularity_linear} to the problem
\begin{equation}
a(\S(\q)-\S(\q+\hq),v;\q)=F(v),\quad\forall v\in\HD{s},
\end{equation}
we conclude that $\S(\q)-\S(\q+\hq)\in\HD{s+\beta}$, where $\beta=\min\{\alpha,1/2-\varepsilon\}$. Since, $\alpha<(1-s)/2-\varepsilon$, it follows that $\beta=\alpha$ and from~\eqref{nonl_regularity} we obtain
\[
\norm{\S(\q)-\S(\q+\hq)}_{\HD{s+\alpha}}\lesssim\norm{F}_{H^{-s+\alpha}(\D)}
\lesssim\norms{h}\norm{\S(\q+\hq)}_{\HD{s+2\alpha}}
\]
that concludes the proof. 
\end{proof}

{We now investigate regularity properties.} By differentiating the state equation~\eqref{eq:var_form} with respect to $\q\in\Q$ in the direction $\hq\in\Q$, and taking into account that $a^\prime_u(u,v;\q)(\hu)=a(\hu,v;\q)$ for any direction $\hu\in\HD{s}$, we obtain the following \emph{sensitivity equation}:
\begin{equation}
a(\hu,v;\q)=-a^{\prime}_\q(u,v;\q)(\hq)\quad\forall v\in\HD{s}.
\label{eq:sensitivity}
\end{equation}
Next, we establish the differentiability of the solution map $\S:\Q\to L^2(\D)$.
\begin{theorem}\label{thm:frechet}
The map $\S:\Q\to L^2(\D)$ is Fr\'echet differentiable, i.e., for $\q\in\Q$ and $\hq\in\Q$ we have
\begin{equation*}
\lim_{|\hq|\to 0}\frac{\norm{\S(\q+\hq)-\S(\q)-\S^\prime(\q)(\hq)}_{\HD{s}}}{|\hq|}=0.
\end{equation*}
\end{theorem}
\begin{proof}
Let $z(\q):=\S(\q+\hq)-\S(\q)-\S^\prime(\q)(\hq)$, then using the similar arguments as in the proof of Proposition~\ref{prop:cont_S}, and the sensitivity equation~\eqref{eq:sensitivity} with $\hu:=\S^\prime(\q)(\hq)$, Lipschitz continuity and boundedness of $a^\prime_q(\cdot,\cdot;\q)$, we obtain:
\begin{equation}
\begin{aligned}
a(z(\q),v;\q)=&\ a(\S(\q+\hq)-\S(\q)-\S^\prime(\q)(\hq),v;\q) \pm a(\S(\q+\hq),v;\q+\hq)\\
=&\ a(\S(q+\hq),v;\q)-a(\S(\q+\hq),v;\q+\hq)-a(\S^\prime(\q)(\hq),v;\q)\\
&\ -\int_0^1 a^\prime_\q(\S(\q+\hq),v;\q+\xi\hq)(\hq)\d\xi+ a^{\prime}_\q(\S(\q),v;\q)(\hq)\\
=&\ \int_{0}^1[ a^\prime_\q(\S(\q+\hq),v;\q)-a^\prime_\q(\S(\q+\hq),v;\q+\xi\hq)](\hq)\d\xi\\
&+ a^{\prime}_\q(\S(\q)-\S(\q+\hq),v;\q)(\hq).
\end{aligned}\end{equation}
Let $u:=S(q+\hq)$, then using~\eqref{lipschitz_der_a} with $\alpha-1/2+\varepsilon\leq\hs<\alpha/2$, where $0<\alpha<\min\{1-s,1/2-\varepsilon\}$, $\varepsilon>0$, and taking into account that $\norms{\hq}\geq\abs{\hs}$ we obtain 
\begin{align*}
&\norms{(a^\prime_s(u,v;\q)-a^\prime_s(u,v;\q+\xi\hq))(\hs)}\\
&\qquad\leq \norms{a^\prime_s(u,v;s,\delta)-a^\prime_s(u,v;s+\xi\hs,\delta+\xi\hd)\pm a^\prime_s(u,v;s+\xi\hs,\delta)}\norms{\hs}\\
&\qquad\leq \norms{a^\prime_s(u,v;s,\delta)- a^\prime_s(u,v;s+\xi\hs,\delta)}\norms{\hs}\\
&\qquad+\norms{a^\prime_s(u,v;s+\xi\hs,\delta)- a^\prime_s(u,v;s+\xi\hs,\delta+\xi\hd)}\norms{\hs}\\
&\qquad\lesssim\xi|\hs|^2\norm{u}_{\HD{s+\alpha}}\norm{v}_{\HD{s}}+|\hs|\norms{\int_{\delta+\xi\hd}^{\delta}a^{\prime\prime}_{s,\delta}(u,v;s+\xi\hs,\eta)\d\eta}\\
&\qquad\lesssim\xi|\hs|^2\norm{u}_{\HD{s+\alpha}}\norm{v}_{\HD{s}}+\xi|\hs||\hd|\norm{u}_{L^2(\D)}\norm{v}_{L^2(\D)}\\
&\qquad\lesssim|\hq|^2\norm{u}_{\HD{s+\alpha}}\norm{v}_{\HD{s}},
\end{align*}
where in the last inequality, under the restrictions on $\hs$ and $\alpha$, stated above, we ensure that $u\in\HD{s+\hs+1/2-\varepsilon}\subset\HD{s+\alpha}$, and hence $\norm{u}_{\HD{s+\alpha}}<\infty$.
In a similar way as above using~\eqref{eq:hodler_der_a_delta} with $\theta=s$, we obtain the estimate with respect to $\delta$:
\begin{align*}
&\norms{(a^\prime_\delta(u,v;\q)-a^\prime_\delta(u,v;\q+\xi\hq))(\hd)}\\
&\qquad\leq \norms{a^\prime_\delta(u,v;s,\delta)- a^\prime_\delta(u,v;s,\delta+\xi\hd)}\norms{\hd}\\
&\qquad+\norms{a^\prime_\delta(u,v;s,\delta+\xi\hd)- a^\prime_\delta(u,v;s+\xi\hs,\delta+\xi\hd)}\norms{\hd}\\
&\qquad\lesssim\xi^s|\hd|^{s+1}\norm{u}_{\HD{s}}\norm{v}_{\HD{s}}
  +|\hd|\norms{\int_{s+\xi\hs}^{s}a^{\prime\prime}_{s,\delta}(u,v;\eta,\delta+\xi\hd)\d\eta}\\
 &\qquad\lesssim|\hq|^{1+s}\norm{u}_{\HD{s+\alpha}}\norm{v}_{\HD{s}}.
\end{align*}
Using the boundedness of $a^\prime_\q(\cdot,\cdot;\q)$,~\eqref{eq:bdd_der_a_delta},~\eqref{eq:bdd_der_a_s}, and the Lipschitz continuity of the map $\S$~\eqref{lipschitz_S}, with further imposing additional conditions on $\hs$ and $\alpha$ such as in Proposition~\ref{prop:cont_S}, we obtain the following:
\begin{align*}
&\norms{a^\prime_\q(\S(\q)-\S(\q+\hq),v;\q)(\hq)}\\
&\qquad\leq\norms{a^\prime_s(\S(\q)-\S(\q+\hq),v;\q)}|\hs|+\norms{a^\prime_\delta(\S(\q)-\S(\q+\hq),v;\q)}|\hd|\\
&{\qquad\lesssim|\hq|\norm{\S(\q)-\S(\q+\hq)}_{\HD{s+\alpha}}\norm{v}_{\HD{s}}\lesssim|\hq|^2\norm{\S(\q+\hq)}_{\HD{s+2\alpha}}\norm{v}_{\HD{s}}}.
\end{align*}
Now, taking $v=\S(\q+\hq)-\S(\q)-\S^\prime(\q)(\hq)$, and using the coercivity of $a(\cdot,\cdot;\q)$, and combining the above estimates we obtain
\begin{multline*}
\norm{\S(\q+\hq)-\S(\q)-\S^\prime(\q)(\hq)}_{\HD{s}}\\
\lesssim |\hq|\left( |\hq|\norm{\S(\q+\hq)}_{\HD{s+\alpha}}+|\hq|^s\norm{\S(\q+\hq)}_{\HD{s+\alpha}}+|\hq|{\norm{\S(\q+\hq)}_{\HD{s+2\alpha}}}\right).
\end{multline*}
Dividing the left-had side by $|\hq|$, and taking the limit $|\hq|\to 0$ leads to the desired result. 
\end{proof}

\section{Optimal control problem}\label{sec:optimization}
In this section we introduce the optimal control problem and analyze the existence of solutions. We augment the matching functional mentioned in the introduction with a regularization term that ensures that the parameters do not attain the values outside of the admissible set. Then, we propose a gradient-based approach that relies on the solution of an adjoint problem and that sets the groundwork for the solution of higher-dimensional inverse problems.

\subsection{The minimization problem}
We define the cost functional as follows:
\begin{equation}
J(u,\q):=\frac{1}{2}\norml{u-\ud}^2+\mathcal{R}(\q),
\label{eq:cost_fun}
\end{equation}
where $\ud\in H^{1/2-\varepsilon}(\D)$ is a given function and {where $\mathcal{R}:\Q\to\R_{+}$ is a regularization term. In particular, $\mathcal{R}(\q)\in C^{2}(\Q)$ is such that} $\mathcal{R}(\q):=\phi_1(s)+\phi_2(\delta)$ with $\phi_i$, $i=1,2$, being given nonnegative convex functions that satisfy
\begin{align}\label{eq:regularization-conditions}
\lim_{s\to 0}\phi_1(s)=+\infty=\lim_{s\to 1}\phi_1(s),\;\; \lim_{\delta\to 0}\phi_{2}(\delta)=+\infty=\lim_{\delta\to \infty}\phi_{2}(\delta).
\end{align}
As such, $\mcR$ prevents the fractional order $s$ to be either $0$ or $1$. {Note that such choice of regularization has been also employed in~\cite{Antil2018sSpectralControl,Sprekels2017}.} 
We formulate the optimal control problem as
\begin{equation}
\min_{\q\in\Q}J(u,\q)\quad\text{subject to }\eqref{eq:var_form}.\label{eq:pr_min}
\end{equation}
By using the control-to-state map we define a reduced cost functional $j(\q)=J(\S(\q),\q)$ and restate~\eqref{eq:pr_min} as an unconstrained minimization problem:
\begin{equation}
\min_{\q\in\Q}j(\q):=\frac{1}{2}\norml{\S(\q)-\ud}^2+\mathcal{R}(\q).\label{eq:pr_min_red}
\end{equation}
Proposition \ref{prop:cont_S} and properties of the cost functional  imply the existence of solutions to the optimization problem, as shown in the following proposition.
\begin{proposition}[Existence]
There exists a solution to the minimization problem~\eqref{eq:pr_min_red}.  
\end{proposition}
\begin{proof}
Since, $j(\q)\geq 0$, it follows that there exists a minimizing sequence $\{\q_n\}_{n\in\N}\subset\Q$, such that $j(\q_n)\to\inf_{\q\in\Q}j(\q)$, $n\to\infty$. Using the definition of $\mathcal{R}$ and $j$, we obtain
\[
\mathcal{R}(\q_n)\leq\frac{1}{2}\norml{\S(\q_n)-\ud}^2+\mathcal{R}(\q_n)\leq C,\quad C>0,
\] 
and hence $\{\q_n\}$ is bounded and we can extract a weakly-convergent sub-sequence $\{\q_{n_k}\}_{n_k\in\N}$, such that $\q_{n_k}\to\hat{\q}\in\overline{\Q}$, $k\to\infty$. Next, we argue that $\hat{\q}\in{\Q}$. Indeed, by continuity we have that $\mathcal{R}(\hat{\q})\leq C$, then if $\hat{\q}\in\overline{\Q}\setminus\Q$, it follows, by the definition of $\mathcal{R}$, that $\mathcal{R}(\q)=+\infty$, that contradicts the previous statement, and, hence $\hat{\q}\in\Q$.  
By the continuity of $\mathcal{R}$ and $\S$, and Proposition~\ref{prop:cont_S}, it follows that $\S(\q_{n_k})\to\S(\hat{\q})$ in $L^2(\D)$ as $k\to\infty$, and we obtain that $j(\q_{n_k})\to j(\hat{\q})$, $k\to\infty$, and 
\[
j(\hat{\q})=\lim_{k\to\infty}j(\q_{n_k})=\inf_{\q\in\Q}j(\q),
\]
thus, $\q$ is the minimizer of~\eqref{eq:pr_min_red}.
\end{proof}

\subsection{Optimality conditions}
Having established the existence of $\S^\prime(\q)(\hq)$ for all $\hq\in\Q$ in Theorem \ref{thm:frechet}, by the chain rule we immediately deduce that $j(\q)$ is G\^ateaux differentiable, and the directional derivative $j^\prime(\q)(\hq)$ can be expressed as follows:
\begin{equation}
j^\prime(\q)(\hq)=(\S(\q)-\ud,\hu)_{L^2(\D)}+\mathcal{R}^\prime(\q),
\label{eq:der_j}
\end{equation}
with $\hu :=\S^\prime(\q)(\hq)$.
\begin{proposition}[First order necessary optimality condition]
If $\q\in\Q$ is a local optimal solution of the optimization problem~\eqref{eq:pr_min_red}, then there holds the first order ne1cessary optimality condition:
\begin{equation}
j^\prime(\q)(\hq)=0,\quad\forall\hq\in\Q.\label{eq:first_optimality}
\end{equation}
\end{proposition}

\subsection{Adjoint approach}
In this section we provide a computable representation of the derivative of $j(\q)$, which is necessary
for evaluating the optimality conditions stated in the previous section, and for employing a gradient-based optimization algorithm. 

For a small dimensional control space, the most obvious choice to evaluate the derivative $j^\prime(\q)$ is the~\emph{sensitivity approach}, where the sensitivity needed to evaluate $j^\prime(\q)$ in~\eqref{eq:first_optimality}, is obtained as a solution of the sensitivity equation~\eqref{eq:sensitivity}. 
While in the present case, {having only a two-dimensional control, would lead us to use the aforementioned procedure, to keep our approach general and set the groundwork for higher-dimensional problems,} we resort to the so-called \emph{adjoint method}.

We define the Lagrangian $\mathcal{L}:\Q\times\HD{s}\times\HD{s}\to\R$ of the optimization problem~\eqref{eq:pr_min_red}:
\begin{equation}
\mathcal{L}(\q,u,z):=J(u,\q)+(f,v)_{L^2(\D)}-a(u,z;\q).\label{eq:Lagrangian}
\end{equation}
Then, the following identity holds true for $u=\S(\q)$ and all $\phi\in\HD{s}$:
\[
j(\q)=J(u,\q)=\mathcal{L}(\q,u,\phi),
\]
and for all $\hq\in\Q$ we obtain 
\begin{equation}
j^\prime(\q)(\hq)=\mathcal{L}^\prime_\q(\q,u,z)(\hq)+\mathcal{L}^\prime_u(\q,u,z)(\hu),\label{eq:der_j1}
\end{equation}
with $\hu:=\S^\prime(\q)(\hq)$. Then the adjoint state $z\in\HD{s}$ is defined as a solution of the adjoint equation
\begin{equation}
\mathcal{L}^\prime_u(\q,u,z)(\phi)=0,\quad\forall\phi\in\HD{s}.\label{eq:adjoindL}
\end{equation}
Taking into account that $\mathcal{L}^\prime_u(\q,u,z)(\hu)=J^\prime_u(u,\q)(\hu)-a(\hu,z;\q)$, the adjoint equation~\eqref{eq:adjoindL} can be written in the following explicit form:
\begin{equation}
a(\phi,z;\q)=(u-\ud,\phi)_{L^2(\D)},\quad\forall\phi\in\HD{s}.
\label{eq:adjoint}
\end{equation}
Combining~\eqref{eq:adjoint} and~\eqref{eq:der_j1}, we obtain the following expression for $j^\prime(\q)$:
\begin{equation}
j^\prime(\q)=\mathcal{L}^\prime_\q(\q,u,z)=\mathcal{R}^\prime(\q)-a^\prime_\q(u,z;\q).
\end{equation}
Invoking Theorem~\ref{thm:regularity_linear} for the state solution $u$, we obtain that $u-\ud\in\HD{1/2-\varepsilon}$. Then, following a boot-strapping procedure, similar to~\cite[Theorem~3.4]{BG2019vi}, it can be shown the solution of the adjoint equation~\eqref{eq:adjoint} admits an improved regularity $z\in\HD{s+1/2-\varepsilon}$, and the following holds
\begin{equation}
\norm{z}_{\HD{s+1/2-\varepsilon}}\lesssim\norm{u-\ud}_{H^{1/2-\varepsilon}(\D)}{\lesssim\norm{f}_{H^{1/2-\varepsilon}(\D)}+\norm{\ud}_{H^{1/2-\varepsilon}(\D)}},\quad\varepsilon>0.\label{eq:adjoint_regularity}
\end{equation}

\section{Approximation of the bilinear form and its derivatives}\label{sec:interp-form}
Based on the work introduced in \cite{BG2019rbm}, in this section we consider an approximation for the bilinear form based on Chebyshev interpolation. This allows for an affine decomposition of the bilinear form and its derivative with respect to the fractional order, which allows us to reduce the computational time needed for the assembly of these forms. At the end of the section we also introduce an approximation of the derivative with respect to the truncation length based on operator splitting.

\subsection{Interpolation of the bilinear form}
We let \(\mathcal{S}=[s_{\min},s_{\max}]\) and consider a subdivision of the interval $\mathcal{S}$ into sub-intervals \(\mathcal{S}_{k}:=[s_{\min,k},s_{\max,k}]\) with \(1\leq k\leq K\), $K\in\N$, and 
\begin{equation}
\begin{aligned}
    s_{\min,1} &= s_{\min},\quad s_{\max,K} = s_{max},\\
    s_{\min,k} &= s_{\max,k-1},  \quad 2\leq k \leq K.
  \end{aligned}\label{eq:subintervals_s}
\end{equation}
Then, we approximate the parametrized bilinear form $a(\cdot,\cdot;\cdot,\cdot)$ with respect to $s$ as follows:
\begin{align*}
  a(u,v;s,\delta) \approx \tilde{a}(u,v;s,\delta) := \sum_{m=0}^{M_k} \Theta_{\mathcal{S}_k,m}(s) a(u,v;s_{m},\delta),
\end{align*}
where \(\Theta_{\mathcal{S}_k,m}(s)\) are Lagrange polynomials with respect to the Chebyshev nodes on the interval \(\mathcal{S}_k\).

Under a certain condition on the size of the interval, which is induced by the limited spatial regularity of the solution, it has been shown in~\cite{BG2019rbm} that the following error estimate for the interpolation of the bilinear form holds true. 
\begin{lemma}[\cite{BG2019rbm}]\label{lem:approximation}
Let \(u\in H^{s_{2,k}}_{\Omega}(\mathbb{R}^{n})\), \(v\in H^{s_{1,k}}_{\Omega}(\mathbb{R}^{n})\) and \(s\in \mathcal{S}_k=[s_{\min,k},s_{\max,k}]\subset(0,1)\), such that \((s_{1,k}+s_{2,k})/2-1/2<s_{\max,k}<(s_{1,k}+s_{2,k})/2\), \(s_{1,k},s_{2,k}\in(0,1)\).
  Then, for \(\delta\in(0,\infty)\), we obtain
  \begin{align*}
    \abs{a(u,v;s,\delta)-\tilde{a}(u,v;s,\delta)} &\leq \sigma_k^{M_{k}+1}C_k(\delta)\norm{u}_{H^{s_{2,k}}_{\Omega}(\mathbb{R}^{n})}\norm{v}_{H^{s_{1,k}}_{\Omega}(\mathbb{R}^{n})},
  \end{align*}
  where \(\sigma_{k}=(s_{\max,k}-s_{\min,k})/(2\hat{\varepsilon}(s_{\max,k}))\), \(C_k(\delta)=4(e^{-1}+\delta^{\hat{\varepsilon}(s_{\min,k})+1})\) if \(\delta>1\) and \(C_k(\delta)=4e^{-1}\) if \(\delta\leq1\), where \(\hat{\varepsilon}(s)=s_{1,k}+s_{2,k}-2s\).
\end{lemma}

Next, we investigate how to choose the intervals \(\mathcal{S}_{k}\) and the interpolation orders \(M_{k}\) in order to satisfy a tolerance $\eta>0$.

\begin{lemma}\label{lem:global-approximation}
  Let \(s\in [s_{\min},s_{\max}]\subset (0,1)\), \(\delta\in(0,\infty)\), and let \(\eta>0\).
  Assume that \(u\in H^{s+1/2-\varepsilon}_{\Omega}(\mathbb{R}^{n})\), \(\varepsilon>0\).
  Then, we define
  \begin{align*}
    s_{\min,1} &= s_{\min}, \quad s_{\max,K} = s_{\max}\\
    s_{\max,k} &= s_{\min,k} + (1/2-\xi) \min\{1-s_{\min,k}, \frac{1}{2}-\varepsilon\},  & 1\leq k\leq K-1,
  \end{align*}
  for some \(\xi\in({1}/{10}, {1}/{2})\).
  Furthermore, for \(s\in\mathcal{S}_{k}\), we set
  \begin{align*}
    M_{k} = \left\lceil \frac{\log (\eta/C_{k}(\delta))}{\log \frac{1}{8}(\xi^{-2}-2)} \right\rceil - 1.
  \end{align*}
  Then, for a given $\xi\in(1/10,1/2)$, $\overline{s}_2(s):=\min\{1,s_{\min,k}+\frac{1}{2}\}-\varepsilon$, \(s\in\mathcal{S}_{k}\) and \(v\in H^{s}_{\Omega}(\mathbb{R}^{n})\), we obtain
  \begin{align*}
    \abs{a(u(s),v;s,\delta)-\tilde{a}(u(s),v;s,\delta)} &\leq \eta \norm{u(s)}_{H^{\overline{s}_{2}(s)}_{\Omega}(\mathbb{R}^{n})}\norm{v}_{H^s_{\Omega}(\mathbb{R}^{n})},
  \end{align*}
where the total number of interpolation nodes satisfies
  \begin{align*}
    \sum_{k=1}^{K}(M_{k}+1) &\leq C \abs{\log \eta},
  \end{align*}
  and the constant $C$ depends on \(\xi\), \(\delta\) and \(s_{\max}\).
\end{lemma}

\begin{proof}
  We will choose \(s_{1,k}\), \(s_{2,k}\), \(s_{\min,k}\) and \(s_{\max,k}\) so that we can apply Lemma~\ref{lem:approximation} on each interval \(\mathcal{S}_{k}\).
  Given \(s_{\min,k}\) for some \(k\), we set
  \begin{align*}
    s_{1,k}&=s_{\min,k},&
    s_{2,k}&=\min\{1,s_{\min,k}+\frac{1}{2}\}-\varepsilon = s_{\min,k} + g_{k},
  \end{align*}
  with $g_{k}:=\min\{1-s_{\min,k},\frac{1}{2}\}-\varepsilon$.
  We then choose
  \begin{align*}
    s_{\max,k}:= \frac{s_{1,k}+s_{2,k}}{2}-\xi g_{k} = s_{\min,k}+\left(\frac{1}{2}-\xi\right)g_{k}.
  \end{align*}
  In order to satisfy \(s_{\min,k}<s_{\max,k}\), we need that \(\xi < {1}/{2}\).
  The contraction factor \(\sigma_{k}\) is then given by \(\sigma_{k}=\frac{s_{\max,k}-s_{\min,k}}{2\hat{\varepsilon}(s_{\max,k})} = (\xi^{-1} - 2)/8\).
  The condition \(\sigma_{k}<1\) implies that \(\xi > {1}/{10}\).
  Since \(g_{k}\geq \min\{1-s_{\max},\frac{1}{2}\}-\varepsilon>0\), the total number of intervals \(K\) is finite.
  Moreover, \(g_{k}\leq 1\), so \(\xi\in({1}/{10},{1}/{2})\) implies that \((s_{1,k}+s_{2,k})/2-1/2<s_{\max,k}<(s_{1,k}+s_{2,k})/2\) is satisfied.

  Now,
  \begin{align*}
    \eta &\leq C_{k}(\delta) \sigma_{k}^{M_{k}+1} \qquad\forall s\in\mathcal{S}_{k}
  \end{align*}
  is satisfied by choosing the interpolation order
  \begin{align*}
    M_{k} &= \left\lceil \frac{\log (\eta/C_{k}(\delta))}{\log \frac{1}{8}(\xi^{-2}-2)} \right\rceil - 1 .
  \end{align*}
  The approximation results then follows from Lemma~\ref{lem:approximation} and the continuous embeddings of \(H^{s}_{\Omega}(\mathbb{R}^{n})\hookrightarrow H^{s_{1,k}}_{\Omega}(\mathbb{R}^{n})\) and \(H^{s+1/2-\varepsilon}_{\Omega}(\mathbb{R}^{n})\hookrightarrow H^{s_{2,k}}_{\Omega}(\mathbb{R}^{n})\) for \(s\in\mathcal{S}_{k}\).

  We have that \(\hat{\varepsilon}(s_{\min,k})=g_{k}\leq 1\), and hence \(C_{k}(\delta)\) can be bounded independently of \(k\).
  Hence, we can bound
  \begin{align*}
    M_{k}\lesssim\frac{\abs{\log\eta}}{\abs{\log \frac{1}{8}(\xi^{-1}-2)}}
  \end{align*}
  independently of \(k\), where the constant depends on \(s_{\max}\) and \(\delta\).
 Since the choice of the intervals is independent of \(\eta\) and the number of intervals \(K\) is finite, the total number of interpolation nodes scales like \(\abs{\log \eta}\).
\end{proof}

\begin{remark}
  While Lemma~\ref{lem:approximation} holds for \(\delta\in(0,\infty)\), the result can be extended to the case of \(\delta=\infty\).
  We can split
  \begin{align*}
    a(u,v;s,\infty) = a_{\Omega}(u,v;s) + a_{\mathbb{R}^{n}\setminus\Omega}(u,v;s),
  \end{align*}
  with
  \begin{align*}
    a_{\Omega}(u,v;s) &= \frac{1}{2}\iint_{\Omega\times\Omega} \frac{(u(\x)-u(\xx))(v(\x)-v(\xx))}{\abs{\x-\xx}^{n+2s}} \d\xx \d\x,\\
    a_{\mathbb{R}^{n}\setminus\Omega}(u,v;s) &= \int_{\Omega} u(\x)v(\x) \int_{\mathbb{R}^{n}\setminus\Omega} \frac{1}{\abs{\x-\xx}^{n+2s}} \d\xx \d\x \\
    &= \int_{\Omega} u(\x)v(\x) \int_{\partial\Omega} \frac{1}{2s}\frac{\n\cdot(\x-\xx)}{\abs{\x-\xx}^{n+2s}} \d\xx \d\x .
  \end{align*}
  We can then apply the approximation to both \(a_{\Omega}\) and \(a_{\mathbb{R}^{n}\setminus\Omega}\) using \(\delta=\operatorname{diam}\Omega\).
  Hence, Lemma~\ref{lem:approximation} and subsequent interpolation results also hold for \(\delta=\infty\).
\end{remark}

\subsection{Interpolation of the \(s\)-derivative}
We discuss the affine approximation of the derivative of the bilinear form with respect to $s$. That is, we consider for a given $u,v\in\HD{s}$ and $\delta\in(0,\infty)$
\begin{align*}
a^{\prime}_s(u,v;s,\delta) &\approx \tilde{a}^{\prime}_s(u,v;s,\delta) = \sum_{m=0}^{M}\Theta^{\prime}_{\mathcal{S}_{k},m}(s) a(u,v;s_{m},\delta), & s\in\mathcal{S}_{k}.
\end{align*}
First, we provide an auxiliary result that establishes the error for the Chebyshev interpolation of the derivative of a function. 
\begin{lemma}\label{lemma:chebyshev_derivative}
Let \(g\in C^{\infty}([-1,1],\mathbb{R})\) and consider an expansion of \(g\) with respect to the Chebyshev polynomials \(T_{m}\) as 
%
%
  \begin{align*}
    g(z)=\sum_{m=0}^{\infty}{\vphantom{\sum}}'\gamma_{m}T_{m}(z),
  \end{align*} 
  where the prime means that the first term is multiplied by $1/2$, and
  \[
  \gamma_{m} = \frac{2}{\pi}\int_{-1}^{1}\frac{g(z)T_m(z)}{\sqrt{1-z^2}}\d z.
  \]
  Let
  \begin{align*}
    (I_{M}g)(z):= \sum_{m=0}^{M} {\vphantom{\sum}}''\beta_{m}T_{m}(z),
  \end{align*}
   be its Chebyshev interpolant of rank \(M\), where the double prime denotes a sum whose first and last terms are multiplied by $1/2$, and 
   \[
   \beta_m = \gamma_m+\sum_{j=1}^{\infty}(\gamma_{m+2jM}+\gamma_{-m+2jm}),\quad 0\leq m\leq M.
   \]
  Then, for $M>2$ the following error estimate holds true
  \begin{align*}
    \abs{g'(z)-(I_{M}g)'(z)} & \leq \max_{\xi\in [-1,1]}|g^{(M+1)}(\xi)|\frac{4}{(M-2)(M-2)!}.
  \end{align*}
\end{lemma}
\begin{proof}
We note that the following relation holds true, see, e.g.,~\cite{trefethen2019approximation},
  \begin{align*}
   g(z)-(I_{M}g)(z) = \sum_{m=M+1}^{\infty}\gamma_m(T_{m}(z)-T_q(z)), 
  \end{align*}
  where $q=q(m,M)=|(m+M-1)({\rm mod}\,2M)-(M-1)|$, and $0\leq q\leq M$. Then, using the above formula and the fact that \(\abs{T_{m}'(z)}\leq m^{2}\), we obtain that
  \begin{align*}
    \abs{g'(z)-(I_{M}g)'(z)}
    &= \abs{\sum_{m=M+1}^{\infty}\gamma_{m}\left(T_{m}'(z) - T^\prime_q(z)\right)} \\
    &\leq \sum_{m=M+1}^{\infty}\abs{\gamma_{m}}\left(\abs{T_{m}'(z)} +\abs{T^\prime_q(z)}\right) \\
    &\leq 2 \sum_{m=M+1}^{\infty}\abs{\gamma_{m}}m^{2}.
  \end{align*}
  We recall that the following bound on the Chebyshev coefficient holds true, see, e.g.,~\cite[Theorem 4.2]{trefethen2008},
  \begin{equation*}
  |\gamma_m|\leq \frac{2 V_k}{\pi m(m-1)\dots(m-k)},\quad m\geq k+1,\quad V_k:=\norm{g^{(k)}}_T=\int_{-1}^1\frac{g^{(k+1)}(z)}{\sqrt{1-z^2}}\d z,
  \end{equation*}
  where $k\geq 1$ denotes the $k$-th derivative of $g$.
  Using the above expressions and the fact that $m/(m-1)<(M+1)/M<2$ for $M<m$ and 
  \[
\sum_{j=M+1}^{\infty}\frac{1}{j(j-1)\dots(j-k)}=\frac{1}{kM(M-1)\dots(M-k+1)}=\frac{(M-k)!}{M!k},
  \]
we can estimate 
  \begin{align*}
 \abs{g'(z)-(I_{M}g)'(z)}&\leq  2 \sum_{m=M+1}^{\infty}\abs{\gamma_{m}}m^{2}\\
& \leq\frac{2 V_{k}}{\pi}\sum_{m=M+1}^{\infty}\frac{m^2}{m(m-1)\dots(m-k)}\\
& \leq\frac{4 V_{k}}{\pi}\sum_{m=M+1}^{\infty}\frac{1}{(m-2)(m-3)\dots(m-k)}\\
&\leq \frac{4 V_{k}}{\pi}\sum_{m=M-1}^{\infty}\frac{1}{m(m-1)\dots(m-k+2)}\\
&= \frac{4 V_{k}(M-k)!}{\pi(k-2)(M-2)!}\\
&\leq 4\max_{\xi\in [-1,1]}|g^{(M+1)}(\xi)|\frac{1}{(M-2)(M-2)!},
  \end{align*}
  where in the last inequality we set $k=M$ and used the fact that
  \[
\abs{V_{k}}\leq\int_{-1}^{1}\frac{\abs{g^{(k+1)}(t)}}{\sqrt{1-t^2}}\d t\leq \pi \max_{\xi\in [-1,1]}|g^{(k+1)}(\xi)|,
  \]
  which concludes the proof.
\end{proof}

Using the previous result we derive the error for the interpolation of the derivative of the bilinear form.

\begin{lemma}\label{lem:derivative_bilform_error}
  Let $u\in\HD{s_{2,k}}$, $v\in\HD{s_{1,k}}$, and $s\in\mathcal{S}_k$, where $s_{1,k}$,$s_{2,k}$ are such that the conditions of Lemma~\ref{lem:approximation} hold true. Then, for $\delta\in(0,\infty)$, $M>2$, we obtain
 \begin{equation}
 \abs{a^{\prime}_s(u,v;s,\delta)-\tilde{a}^\prime_{s}(u,v;s,\delta)}\lesssim
\tilde{\eta}_k\norm{u}_{\HD{s_{2,k}}}\norm{v}_{\HD{s_{1,k}}},\label{eq:derivative_bilform_error}
\end{equation}
where $\tilde{\eta}_k=\tilde{\sigma}_k^M {M^2}$\ with $\tilde{\sigma}_k=(s_{\max,k}-s_{\min,k})/\hat{\varepsilon}(s_{\max,k})$, and $\hat{\varepsilon}(s)$ is defined as in Lemma~\ref{lem:approximation}.
\end{lemma}
\begin{proof}
We present a proof for $\delta>{\rm diam}|\D|$, the case $\delta\leq{\rm diam}|\D|$ follows from the same arguments. Using  the transformation $s={(s_{\min,k}+s_{\max,k})}/{2}+z(s_{\max,k}-s_{\min,k})/2$,  $z\in[-1,1]$, and applying Lemma~\ref{lemma:chebyshev_derivative}, we obtain
\begin{multline*}
  \abs{a^{\prime}_s(u,v;s,\delta)-\tilde{a}^\prime_{s}(u,v;s,\delta)}\leq \\
  \frac{4(s_{\max,k}-s_{\min,k})^M}{2^M(M-2)(M-2)!}\max_{\xi\in[s_{\min,k},s_{\max,k}]}\abs{a^{(M+1)}_s(u,v;\xi,\delta)}.
\end{multline*}
Using the same steps as in the proof of~\cite[Lemma~6.2]{BG2019rbm}, we obtain that
\[
\max_{\xi\in[s_{\min,k},s_{\max,k}]}\abs{a^{(M+1)}_s(u,v;\xi,\delta)}\lesssim \frac{2^{M+1}(M+1)!}{\hat{\varepsilon}(s_{\max,k})^{M+1}}\norm{u}_{\HD{s_{2,k}}}\norm{v}_{\HD{s_{1,k}}}.
\]
Then, combining the previous two estimates we obtain
\begin{multline*}
 \abs{a^{\prime}_s(u,v;s,\delta)-\tilde{a}^{\prime}_{s}(u,v;s,\delta)}\lesssim \\
 \left(\frac{s_{\max,k}-s_{\min,k}}{\hat{\varepsilon}(s_{\max,k})}\right)^M\cdot\frac{(M+1)!}{(M-2)(M-2)!\hat{\varepsilon}(s_{\max,k})}\norm{u}_{\HD{s_{2,k}}}\norm{v}_{\HD{s_{1,k}}}\\
 \lesssim\tilde{\sigma}_k^M {M^2}\norm{u}_{\HD{s_{2,k}}}\norm{v}_{\HD{s_{1,k}}},
\end{multline*}
which concludes the proof.
\end{proof}

\subsection{Operator splitting for \(\delta<\infty\) and \(\delta\)-derivative}

While the bilinear form \(a(\cdot,\cdot;s,\delta)\) exhibits a smooth behavior with respect to \(s\), as shown in Proposition~\ref{prop:s_reg}, it is less regular with respect to $\delta$, inherited by the limited spatial regularity of the solution; in fact, only a single derivative with respect to \(\delta\) is available (Proposition~\ref{prop:delta_reg}).
Therefore, \(\delta\)-interpolation of the operator is bound to lead to unsatisfactory results.
Instead, we split the operator into
\begin{align*}
  a(u,v;s,\delta)
  &= a(u,v;s,\infty) + c(u,v;s,\delta),
\end{align*}
with correction term
\begin{align}
  c(u,v;s,\delta)
  &=a(u,v;s,\delta) - a(u,v;s,\infty) \nonumber \\
  &= -\int_{\R^{n}}\int_{\R^{n}} \frac{(u(\x)-u(\xx))(v(\x)-v(\xx))}{\abs{\x-\xx}^{n+2s}} \mcX_{\abs{\x-\xx}>\delta}\d\xx\d\x \nonumber \\
  &= -2 \left(u,v\right)_{L^{2}(\Omega)} \int_{\R^{n}\setminus B_{\delta}(0)}\frac{1}{\abs{\xx}^{n+2s}}\d\xx  + 2\int_{\Omega}\int_{\Omega} \frac{u(\x) v(\xx)}{\abs{\x-\xx}^{n+2s}} \mcX_{\abs{\x-\xx}>\delta}\d\xx\d\x \nonumber \\
  &= -\frac{2\pi^{n/2}}{\Gamma(n/2)}\frac{\delta^{-2s}}{s} \left(u,v\right)_{L^{2}(\Omega)} + 2\int_{\Omega}\int_{\Omega} \frac{u(\x) v(\xx)}{\abs{\x-\xx}^{n+2s}} \mcX_{\abs{\x-\xx}>\delta}\d\xx\d\x. \label{eq:correction}
\end{align}
{We note that for \(\delta\geq \operatorname{diam}\Omega\) the last term of \eqref{eq:correction} is zero. As before, we can interpolate the bilinear form in \(s\) as follows}
\begin{align*}
  a(u,v;s,\delta) &\approx \tilde{a}(u,v;s,\infty) + c(u,v;s,\delta)  \\
  &= \sum_{m=0}^{M_k} \Theta_{\mathcal{S}_k,m}(s) a(u,v;s_{m},\infty) + c(u,v;s,\delta), \quad s\in\mathcal{S}_{k}.
\end{align*}
Changing the value of \(\delta\) only requires reassembly of the correction term \(c(\cdot,\cdot;s,\delta)\), which is significantly less expensive than assembly of \(a(\cdot,\cdot;s_{m},\delta)\), since integration over the singularity \(\x=\xx\) is avoided.

Moreover, we find that the \(\delta\)-derivative is given by
\begin{align*}
  a^{\prime}_{\delta}(u,v;s,\delta)
  &=c^{\prime}_{\delta}(u,v;s,\delta)\\
  &=\frac{4\pi^{n/2}}{\Gamma(n/2)} \delta^{-1-2s} \left(u,v\right)_{L^{2}(\Omega)} -2\delta^{-n-2s}\int_{\Omega} u(\x)\int_{\partial B_{\delta}(\x)}v(\xx) \d\xx \d\x \\
  &=\frac{4\pi^{n/2}}{\Gamma(n/2)} \delta^{-1-2s} \left(u,v-\overline{v}\right)_{L^{2}(\Omega)},
\end{align*}
where
\begin{align*}
  \overline{v}(\x)&=\frac{1}{\abs{\partial B_{\delta}(\x)}}\int_{\partial B_{\delta}(\x)}v(\xx) \d\xx.
\end{align*}
Note that, as described in the following section, we rely on finite element discretizations and that for finite element functions $\overline{v}$ can be evaluated exactly. We also observe that \(\overline{v}=0\) for \(\delta\geq \operatorname{diam}\Omega\).

\section{Discretization and error estimates}\label{sec:discretization-error-estimates}
In this section we first recall results of finite element approximations for the state and adjoint variables and then provide discretization error estimates for the state and adjoint solutions obtained via interpolated forms.
We also provide, for the interpolated problem, error estimates for the gradient of the cost functional.

\medskip
In what follows, we let $\D\subset\R^n$, $n=1,2$ be convex, and let \(\mathcal{T}_{h}\) be a family of shape-regular and locally quasi-uniform triangulations of \(\Omega\) \cite{ErnGuermond2004_TheoryPracticeFiniteElements}, and let \(\mathcal{N}_{h}\) be the set of vertices of \(\mathcal{T}_{h}\), \(h_{K}\) be the diameter of the element \(K\in\mathcal{T}_{h}\).
Let \(\phi_{i}\) be the usual piecewise linear Lagrange basis function associated with a node \(\vec{z}_{i}\in\mathcal{N}_{h}\), satisfying \(\phi_{i}\left(\vec{z}_{j}\right)=\delta_{ij}\) for \(\vec{z}_{j}\in\mathcal{N}_{h}\), and let \(X_{h}:=\operatorname{span}\left\{\phi_{i}\mid \vec{z}_{i}\in\mathcal{N}_{h}\right\}\).
The finite element subspace \(V_{h}\subset H_{\Omega}^{s}(\mathbb{R}^{n})\) is given by
\begin{align*}
  V_{h} = \left\{v_{h}\in X_{h}\mid v_{h}=0 \text{ on }\partial\Omega\right\} = \operatorname{span}\left\{\phi_{i}\mid \vec{z}_{i}\not\in \partial\Omega\right\}.
\end{align*}

\subsection{Error estimates for the state and adjoint equations}

It is well known that the following inverse inequality holds for {$0\leq \beta\leq \alpha$, see, e.g.,~\cite[Corollary 1.141]{ErnGuermond2004_TheoryPracticeFiniteElements},
\begin{align}
  \norm{v_{h}}_{\HD{\alpha}}\leq C_{I}h^{\beta-\alpha}\norm{v_{h}}_{\HD{\beta}} &&\forall v_{h}\in V_{h}.\label{eq:inverse_inequality}
\end{align}}
Let \(I_{h}\) be the Scott-Zhang interpolation operator \cite[Section 1.6.2]{ErnGuermond2004_TheoryPracticeFiniteElements} satisfying for \(\beta\in[0,2-s]\) the approximation property
\begin{align}
\norm{v-I_{h}v}_{\HD{s}} & \lesssim h^{\beta}\norm{v}_{\HD{s+\beta}} \quad \forall v\in \HD{s+\beta} \label{eq:interpolation-error}
\end{align}
and for \(\beta\in[0,1]\) the stability property
\begin{align}
\norm{I_{h}v}_{\HD{\beta}}&\lesssim \norm{v}_{\HD{\beta}} \quad \forall v\in \HD{\beta}. \label{eq:stability-interpolation}
\end{align}
Then, for $f\in H^{1/2-\varepsilon}(\D)$, the discrete version of~\eqref{eq:var_form} reads as follows: find $u_h\in V_h$ such that
\begin{equation}
a(u_h,v_h;\q)=(f,v_h)_{L^2(\D)},\quad\forall v_h\in V_h.\label{eq:state_discrete}
\end{equation}
From~\cite[Proposition~3.5]{BG2019vi}, we obtain the following a-priori error estimate
\begin{align}
&\norm{u-u_h}_{\HD{s}}\lesssim h^{1/2-\varepsilon}\norm{f}_{\HD{1/2-\varepsilon}},\label{eq:error_state_H}\\
 &\norm{u-u_h}_{L^2(\D)}\lesssim h^{1/2-\varepsilon+\beta}\norm{f}_{\HD{1/2-\varepsilon}}\label{eq:error_state_L2},
\end{align}
where $\beta=\min\{s,1/2-\varepsilon\}$ and the constant $C$ is independent of $h$. Similarly, we consider a discrete version of the adjoint equation~\eqref{eq:adjoint}: find $z_h\in V_h$ such that
\begin{equation}
a(\phi_h,z_h;\q)=(u_h-\ud,\phi_h)_{L^2(\D)},\quad\forall \phi_h\in V_h.\label{eq:adjoint_discrete}
\end{equation}
The following proposition establishes the a priori error estimates for the adjoint equation.
\begin{proposition}
For $\q\in\Q$, let $z(\q)\in\HD{s+1/2-\varepsilon}$ and $z_h(\q)\in V_h$ be the solutions of~\eqref{eq:adjoint} and~\eqref{eq:adjoint_discrete}, respectively. Then, the following holds
\begin{equation}
\norm{z-z_h}_{\HD{s}}\lesssim h^{1/2-\varepsilon}\left(\norm{f}_{H^{1/2-\varepsilon}(\D)}+\norm{\ud}_{H^{1/2-\varepsilon}(\D)}\right).
\end{equation}
\end{proposition}
\begin{proof}
Let $\ud\in\HD{1/2-\varepsilon}$ and let $u\in\HD{s}$ be the solution of~\eqref{eq:var_form}; we denote by $\hat{z}_h\in\HD{s}$ the solution of  
\[
a(\phi_h,\hat{z}_h;\q)=(u-\ud,\phi_h)_{L^2(\D)},\quad\forall\phi_h\in V_h.
\]
By the standard Lax-Milgram argument the above problem is well-posed. Furthermore, since $u-\ud\in\HD{1/2-\varepsilon}$, by invoking the regularity result~\eqref{eq:adjoint_regularity}, it follows from~\cite[Proposition~3.5]{BG2019vi} that 
\[
\norm{z-\hat{z}_h}_{\HD{s}}\lesssim h^{1/2-\varepsilon}\norm{u-\ud}_{H^{1/2-\varepsilon}(\D)}.\]
By applying the above estimate and \eqref{eq:error_state_L2} we obtain
\begin{align*}
\norm{z-z_h}_{\HD{s}}\leq&\ \norm{z-\hat{z}_h}_{\HD{s}}+\norm{\hat{z}_h-z_h}_{\HD{s}}\\
\lesssim& h^{1/2-\varepsilon}\norm{u-\ud}_{H^{1/2-\varepsilon}(\D)}+\norm{u-\ud}_{H^{-s}(\D)}\\
\lesssim&\  h^{1/2-\varepsilon}\norm{u-\ud}_{H^{1/2-\varepsilon}(\D)}+h^{1/2-\varepsilon+\beta}\norm{f}_{H^{1/2-\varepsilon}(\D)}\\
\lesssim&\ h^{1/2-\varepsilon}\left(\norm{f}_{H^{1/2-\varepsilon}(\D)}+\norm{\ud}_{H^{1/2-\varepsilon
}(\D)}\right).
\end{align*}
This concludes the proof.
\end{proof}

\subsection{Error estimates for the interpolated state and adjoint equations}
{Note that, in practice, instead of solving~\eqref{eq:state_discrete} and~\eqref{eq:adjoint_discrete}, which involve the exact bilinear form, we consider the} approximated problem where the bilinear form is replaced by its affine interpolant. That is, we solve the discretized problems of finding, respectively, $\tilde{u}_{h}(\q)\in V_{h}$, such that
\begin{equation}
  \tilde{a}(\tilde{u}_{h}(\q),v_{h};\q)=(f,v_{h})_{L^2(\D)},\quad\forall v_{h}\in V_{h},\label{eq:discrete_var_form}
\end{equation}
and $\tilde{z}_h$ such that 
\begin{equation}
  \tilde{a}(\phi_h,\tilde{z}_{h};\q)=(\tilde{u}_h-\ud,\phi_{h})_{L^2(\D)},\quad\forall \phi_{h}\in V_{h}.\label{eq:adjoint_discrete_affine}
\end{equation} 
The next theorem provides the error estimate for the state solution of the interpolated problem.
\begin{theorem}\label{thm:solution_error_affine}
  Let \(\q=(s,\delta)\in[s_{\min},s_{\max}]\times(0,\infty)\), and let \(u(\q)\) be the solution of \eqref{eq:var_form}.
  We consider an operator interpolation as given by Lemma~\ref{lem:global-approximation} for the tolerance
  \begin{align}
    \eta &\leq \frac{\min_{s\in[s_{\min},s_{\max}]}\alpha(s,\delta)}{2C_{I}}h^{1/2-\varepsilon}, \label{eq:error-bound-interpolation}
  \end{align}
  where \(\alpha(s,\delta)\) is the coercivity constant of \(a(\cdot,\cdot;\q)\).
  Let also $\tilde{u}_h$ be the solution of the interpolated problem~\eqref{eq:discrete_var_form}, then
  \begin{align}
    \norm{u-\tilde{u}_{h}}_{\HD{s}} &\lesssim \left(h^{1/2-\varepsilon}+\eta\right)\norm{f}_{H^{1/2-\varepsilon}(\D)}\lesssim h^{1/2-\varepsilon}\norm{f}_{H^{1/2-\varepsilon}(\D)}. \label{eq:solution_error_affine}
  \end{align}
\end{theorem}
\begin{proof}
  For \(\tilde{u}_{h}\in V_{h}\)
  \begin{align*}
    \tilde{a}(\tilde{u}_{h},\tilde{u}_{h};\q)
    & \geq a(\tilde{u}_{h},\tilde{u}_{h};\q) - \abs{a(\tilde{u}_{h},\tilde{u}_{h};\q)-\tilde{a}(\tilde{u}_{h},\tilde{u}_{h};\q)} \\
    &\geq {\alpha(s,\delta)} \norm{\tilde{u}_{h}}_{\HD{s}}^{2} - \eta \norm{\tilde{u}_{h}}_{\HD{\overline{s}_{2}(s)}}\norm{\tilde{u}_{h}}_{\HD{s}} \\
    &\geq \left({\alpha(s,\delta)} - \frac{C_{I}\eta}{h^{\overline{s}_{2}(s)-s}}\right)\norm{\tilde{u}_{h}}_{\HD{s}}^{2} \\
    &= \left({\alpha(s,\delta)} - \frac{\min\alpha(s,\delta)}{2} h^{s+1/2-\varepsilon-\overline{s}_{2}(s)}\right)\norm{\tilde{u}_{h}}_{\HD{s}}^{2} \\
    &\geq\left({ \alpha(s,\delta)}/2 \right)\norm{\tilde{u}_{h}}_{\HD{s}}^{2},
  \end{align*}
  where we have used the coercivity of \(a(\cdot,\cdot;\q)\), Lemma~\ref{lem:global-approximation}, the inverse inequality~\eqref{eq:inverse_inequality}, and \(s+1/2-\varepsilon-\overline{s}_{2}(s)\geq 0\), where $\overline{s}_2(s)$ is defined as in Lemma~\ref{lem:global-approximation}. This shows that \(\tilde{a}(\cdot,\cdot;\q)\) is coercive on the discrete space \(V_{h}\).

  By Strang's first lemma \cite{ErnGuermond2004_TheoryPracticeFiniteElements}, we have that
  \begin{align*}
    \norm{u-\tilde{u}_{h}}_{\HD{s}}
    &\lesssim \inf_{v_{h}} \left\{\norm{u-v_{h}}_{\HD{s}} + \sup_{w_{h}}\frac{\abs{a(v_{h},w_{h};\q)-\tilde{a}_{\xi}(v_{h},w_{h};\q)}}{\norm{w_{h}}_{\HD{s}}} \right\} \\
    &\lesssim \inf_{v_{h}} \left\{\norm{u-v_{h}}_{\HD{s}} + \eta \norm{v_{h}}_{\HD{\overline{s}_{2}(s)}}  \right\} \\
    &\lesssim \left\{\norm{u-I_{h}u}_{\HD{s}} + \eta \norm{I_{h}u}_{\HD{\overline{s}_{2}(s)}}  \right\}.
  \end{align*}
  Using \eqref{eq:interpolation-error} and \eqref{eq:stability-interpolation} we obtain
  \begin{align*}
    \norm{u-\tilde{u}_{h}}_{\HD{s}} &\lesssim h^{1/2-\varepsilon}\norm{u}_{\HD{s+1/2-\varepsilon}} + \eta \norm{u}_{\HD{\overline{s}_{2}}} \nonumber\\
    &\lesssim(h^{1/2-\varepsilon}+\eta)\norm{u}_{\HD{s+1/2-\varepsilon}}.
  \end{align*}
  Finally, using \eqref{eq:error-bound-interpolation}, Assumption~\ref{ass:f-regularity} and Theorem~\ref{thm:regularity_linear}  we obtain the desired result.
\end{proof}

Next, we derive the error estimates for the interpolated adjoint problem.
\begin{theorem}\label{thm:adjoint_error_affine}
For $\q\in\Q$, let $z(\q)\in\HD{s+1/2-\varepsilon}$ and $z_h(\q)\in V_h$ be the solutions of~\eqref{eq:adjoint} and~\eqref{eq:adjoint_discrete}, respectively, and the tolerance $\eta$ for the operator interpolation is chosen as in Theorem~\ref{thm:solution_error_affine}. Then, the following holds
  \begin{align}
    \norm{z-\tilde{z}_{h}}_{H_{\Omega}^{s}} &\lesssim \left(h^{1/2-\varepsilon}+\eta\right)\left(\norm{f}_{H^{1/2-\varepsilon}(\D)}+\norm{\ud}_{H^{1/2-\varepsilon}(\D)}\right).\label{eq:error_adjoint_affine}
  \end{align}
\end{theorem}
\begin{proof}
The proof follows similar arguments as in Theorem~\ref{thm:solution_error_affine} by using  Strang's first lemma and invoking~\eqref{eq:inverse_inequality},~\eqref{eq:interpolation-error}, and~\eqref{eq:solution_error_affine}.
\end{proof}

\subsection{Error estimates for the gradient of the cost functional}
First, we prove auxilliary convergence results, which will be needed later to derive the error estimate for the derivative of the reduced cost functional. 
\begin{lemma}\label{lemma:aux_error_estimate}
Let $u,\tilde{u}_h$ and $z,\tilde{z}_h$ be the continuous and discrete solutions of~\eqref{eq:var_form},~\eqref{eq:discrete_var_form} and~\eqref{eq:adjoint},~\eqref{eq:adjoint_discrete_affine}, respectively, and let the conditions of Theorem~\ref{thm:solution_error_affine} and Theorem~\ref{thm:adjoint_error_affine} hold true. Then, for $\mu\in [0,s+1/2)$, $\varepsilon>0$, and $\alpha:=\min\{s-\mu,1/2-\varepsilon\}$, $s\in (0,1)$ the following holds true
\begin{align}
\norm{u-\tilde{u}_h}_{\HD{\mu}}\lesssim
\left(h^{\min\{s+1/2-\mu,1-\varepsilon\}-\varepsilon}+\eta\right)\norm{f}_{H^{1/2-\varepsilon}(\D)}\label{eq:state_error_1}
\end{align}
and
\begin{align}
\norm{z-\tilde{z}_h}_{\HD{\mu}}\lesssim
\left(h^{\min\{s+1/2-\mu,1-\varepsilon\}-\varepsilon}+\eta\right)C_d
\label{eq:adj_error_1}
\end{align}
where $C_d:=\norm{f}_{H^{1/2-\varepsilon}(\D)}+\norm{\ud}_{H^{1/2-\varepsilon}(\D)}$, and $\HD{0}\equiv L^2(\D)$.  Furhermore, we have the following norm bound 
\begin{equation}
\norm{\tilde{u}_h}_{\HD{\mu}}\lesssim\norm{f}_{H^{1/2-\varepsilon}(\D)},
\label{eq:state_norm_1}
\end{equation}
for \(\mu\in[0,s+1/2)\), with the hidden constant independent of $h$.
\end{lemma}
\begin{proof}
We present a proof for the state solutions only, since for the adjoint solutions it follows similar arguments and we omit it for compactness purposes.
First, we consider the case $\mu\in[s,s+1/2)$. Using~\eqref{eq:inverse_inequality},~\eqref{eq:interpolation-error},~\eqref{nonl_regularity},~\eqref{eq:solution_error_affine} we obtain
\begin{multline*}
\norm{u-\tilde{u}_h}_{\HD{\mu}}\leq\norm{u-I_hu}_{\HD{\mu}}+\norm{I_hu-\tilde{u}_h}_{\HD{\mu}}\\
\lesssim h^{s+1/2-\varepsilon-\mu}\norm{u}_{\HD{s+1/2-\varepsilon}}
+h^{s-\mu}\norm{I_hu-\tilde{u}_h}_{\HD{s}}\\
\lesssim h^{s+1/2-\varepsilon-\mu}\norm{f}_{H^{1/2-\varepsilon}(\D)}
+h^{s-\mu}\left(\norm{I_hu-u}_{\HD{s}}+\norm{u-\tilde{u}_h}_{\HD{s}}\right)\\
\lesssim h^{s+1/2-\varepsilon-\mu}\norm{f}_{H^{1/2-\varepsilon}(\D)}.
\end{multline*}
Then, it immediately follows 
that 
\[
\norm{\tilde{u}_h}_{\HD{\mu}}\leq\norm{\tilde{u}_h-u}_{\HD{\mu}}+\norm{u}_{\HD{s}}\lesssim (h^{s+1/2-\varepsilon-\mu} +1)\norm{f}_{H^{1/2-\varepsilon}(\D)},\] 
and using $h\lesssim 1$ we conclude~\eqref{eq:state_norm_1} .
Next, using a duality argument we prove the case $\mu\in[0,s)$. For $\chi\in\HD{s}$ we consider the following problem
\[
a(v,\chi;\q)=l(v)\quad\forall v\in\HD{s},
\]
where $l(v):=(u-\tilde{u}_h,v)_{\HD{\mu}}$. It is clear that $l\in H^{-\mu}(\D)$, and applying Theorem~\ref{thm:regularity_linear} we obtain that $\chi\in\HD{s+\alpha}$ with $\alpha=\min\{s-\mu,1/2-\varepsilon\}$, and 
\begin{equation}
\norm{\chi}_{\HD{s+\alpha}}\lesssim \norm{u-\tilde{u}_h}_{\HD{\mu}}.\label{eq:proof_smth1}
\end{equation}
Now choosing $v=u-\tilde{u}_h\in\HD{s}$, and using Lemma~\ref{lem:global-approximation}, Theorem~\ref{thm:solution_error_affine},~\eqref{eq:interpolation-error},~\eqref{eq:stability-interpolation}, and taking into account that $a(u,I_h\chi;\q)-\tilde{a}(\tilde{u}_h,I_h\chi)=0$, we obtain
\begin{multline*}
\norm{u-\tilde{u}_h}^2_{\HD{\mu}}
=a(u-\tilde{u}_h,\chi;\q) - a(u,I_h\chi;\q)+\tilde{a}(\tilde{u}_h,I_h\chi;\q) \\
\leq\abs{a(u-\tilde{u}_h,\chi-I_h\chi;\q)}+\abs{a(\tilde{u}_h,I_h\chi;\q)-\tilde{a}(\tilde{u}_h,I_h\chi;\q)}\\
\leq \norm{u-\tilde{u}_h}_{\HD{s}}\norm{\chi-I_h\chi}_{\HD{s}}+\eta\norm{\tilde{u}_h}_{\HD{\bar{s}_2(s)}}\norm{I_h\chi}_{\HD{s}}\\
\lesssim h^{\alpha+1/2-\varepsilon}\norm{\chi}_{\HD{s+\alpha}}\norm{f}_{H^{1/2-\varepsilon}(\D)}+\eta\norm{\tilde{u}_h}_{\HD{s+1/2-\varepsilon}}\norm{\chi}_{\HD{s}}\\
\lesssim \left(h^{\alpha+1/2-\varepsilon}+\eta\right)\norm{u-\tilde{u}_h}_{\HD{\mu}}\norm{f}_{H^{1/2-\varepsilon}(\D)},
\end{multline*}
where the last two inequalities have been obtained by using that $\tilde{u}_h\in\HD{s+1/2-\varepsilon}\hookrightarrow\HD{\bar{s}_2(s)}$, where $\bar{s}_2(s)$ such as in Lemma~\ref{lem:global-approximation}, and $\norm{\tilde{u}_h}_{\HD{s+1/2-\varepsilon}}\lesssim\norm{f}_{H^{1/2-\varepsilon}}$ by~\eqref{eq:state_norm_1}, and invoking the stability estimate 
\[
\norm{\chi}_{\HD{s}}\lesssim\norm{u-\tilde{u}_h}_{\HD{\mu}}
\]
together with~\eqref{eq:proof_smth1}. Then, dividing by $\norm{u-\tilde{u}_h}_{\HD{\mu}}$ we obtain the desired result~\eqref{eq:state_error_1} for $\mu\in[0,s)$. The remaining norm bound~\eqref{eq:state_norm_1} is obtained by the similar arguments as before. This concludes the proof. 
\end{proof}

For $\tilde{u}_h,\tilde{z}_h\in V_h$ the solutions of~\eqref{eq:discrete_var_form} and~\eqref{eq:adjoint_discrete_affine}, respectively, we define
\[
\tilde{j}_h^{\prime}(\q):=\mathcal{R}^\prime(\q)-\tilde{a}^\prime_q(\tilde{u}_h,\tilde{z}_h;\q). 
\]
Then, we obtain the following error estimate for the derivative of the reduced cost functional.

\begin{theorem}\label{thm:objective_error_affine}
For $\q\in\Q$, $s\in(0,1)$, $\varepsilon>0$ the following estimate holds
  \begin{align}
    \abs{j^\prime(\q)-\tilde{j}^\prime_h(\q)}\lesssim\begin{dcases}
   |\log h|\left(h^{s+1/2-\varepsilon}+\eta\right)+\tilde{\eta}\quad &s<1/2,\\
      h^{3/2-s-3\varepsilon}+\eta+\tilde{\eta},\quad &s\geq 1/2,
    \end{dcases}\label{eq:error_obj_derivative}
  \end{align}
  where $\eta$ is defined as in Theorem~\ref{thm:solution_error_affine}, $\tilde{\eta}=\max_{k}\tilde{\eta}_k$ and $\tilde{\eta}_k$ is defined as in~\eqref{eq:derivative_bilform_error}.
\end{theorem}
\begin{proof}
We consider
\begin{multline}
\abs{j^\prime(\q)-\tilde{j}^\prime_h(\q)}\leq\abs{a^{\prime}_\q(u,z;\q)-\tilde{a}^{\prime}_\q(\tilde{u}_h,\tilde{z}_h;\q)}\\
\leq \abs{a^\prime_s(u-\tilde{u}_h,z;\q)}
+\abs{a^\prime_s(\tilde{u}_h,z-\tilde{z}_h;\q)}
+\abs{a^\prime_s(\tilde{u}_h,\tilde{z}_h;\q)-\tilde{a}^\prime_s(\tilde{u}_h,\tilde{z}_h;\q)}\\
+\abs{a^\prime_\delta(u-\tilde{u}_h,z;\q)}+\abs{a^\prime_\delta(\tilde{u}_h,z-\tilde{z}_h;\q)}.\label{eq:proof_redcost_1}
\end{multline}
Next, we estimate each term from the above.  For the first two terms we distinguish between two cases $s<1/2$ and $s\geq 1/2$. In particular, for $s\in(0,1/2)$, using~\eqref{eq:bdd_der_a_s} with $s_1=2s$ and $s_2=\xi$, $0<\xi<2(1-s)$,~\eqref{eq:adjoint_regularity},~\eqref{eq:inverse_inequality}, and~Lemma~\ref{lemma:aux_error_estimate}, we can estimate
\begin{multline*}
\abs{a^\prime_s(u-\tilde{u}_h,z;\q)}\leq C(1,\xi)\norm{u-\tilde{u}_h}_{\HD{\xi}}\norm{z}_{\HD{2s}}\\
\leq C(1,\xi)\left(\norm{u-I_h u}_{\HD{\xi}}+\norm{I_hu-\tilde{u}_h}_{\HD{\xi}} \right)\norm{z}_{\HD{2s}}\\
\lesssim C(1,\xi)\left(h^{s+1/2-\varepsilon-\xi}+h^{-\xi}\norm{I_hu-\tilde{u}_h}_{L^2(\D)}\right)\norm{z}_{\HD{2s}}\\
\lesssim C(1,\xi)h^{-\xi}\left(h^{s+1/2-\varepsilon}+\norm{I_hu-{u}}_{L^2(\D)}+\norm{u-\tilde{u}_h}_{L^2(\D)}\right)\norm{z}_{\HD{2s}}\\
\lesssim|\log h|\left(h^{s+1/2-\varepsilon}+\eta\right),
\end{multline*}
where the last inequality has been obtained by taking $\xi\sim -1/\log h$, and noting that $C(1,\beta)$ in~\eqref{eq:bdd_der_a_s} scales like $C(1,\beta)\sim \xi^{-1}$, and, hence, $C(1,\xi)h^{-\xi}\sim|\log h|$. Similarly, we estimate the second term in~\eqref{eq:proof_redcost_1} for $s\in(0,1/2)$ and obtain
\begin{multline*}
\abs{a^\prime_s(\tilde{u}_h,z-\tilde{z}_h;\q)}\leq C(1,\xi)\norm{z-\tilde{z}_h}_{\HD{\xi}}\norm{\tilde{u}_h}_{\HD{2s}}\\
\lesssim C(1,\xi)h^{-\xi}\left(h^{s+1/2-\varepsilon}\norm{z}_{\HD{s+1/2-\varepsilon}}+\norm{I_h z-\tilde{z}_h}_{L^2(\D)}\right)
\norm{\tilde{u}_h}_{\HD{2s}}\\
\lesssim|\log h|\left(h^{s+1/2-\varepsilon}+\eta\right), 
\end{multline*}
where we have invoked Lemma~\ref{lemma:aux_error_estimate} to estimate $\norm{\tilde{u}_h}_{\HD{2s}}\lesssim\norm{f}_{H^{1/2-\varepsilon}(\D)}$. 
\\
Next, we consider the case $s\in[1/2,1)$. As previously, using~\eqref{eq:bdd_der_a_s} with $s_1=1-\varepsilon$ and $s_2=2s-1+2\varepsilon$,~\eqref{eq:adjoint_regularity},~\eqref{eq:inverse_inequality}, and~Lemma~\ref{lemma:aux_error_estimate} for $\mu=2s-1+2\varepsilon$ we obtain
\begin{equation*}
\abs{a^\prime_s(u-\tilde{u}_h,z;\q)}\lesssim \norm{u-\tilde{u}_h}_{\HD{2s-1+2\varepsilon}}\norm{z}_{\HD{1-\varepsilon}}
\lesssim h^{3/2-s-3\varepsilon}+\eta.
\end{equation*}
Similarly, it follows for the second term in~\eqref{eq:proof_redcost_1}
\begin{equation*}
\abs{a^\prime_s(\tilde{u}_h,z-\tilde{z}_h;\q)}\lesssim \norm{z-\tilde{z}_h}_{\HD{2s-1+2\varepsilon}}\norm{\tilde{u}_h}_{\HD{1-\varepsilon}}
\lesssim h^{3/2-s-3\varepsilon}+\eta.
\end{equation*}
Next, using~\eqref{eq:bdd_der_a_delta},~\eqref{eq:poincare} and Lemma~\ref{lemma:aux_error_estimate} we can estimate the last two terms in~\eqref{eq:proof_redcost_1}
\begin{multline*}
\abs{a^\prime_\delta(u-\tilde{u}_h,z;\q)}+\abs{a^\prime_\delta(\tilde{u}_h,z-\tilde{z}_h;\q)}\\
\lesssim\norm{u-\tilde{u}_h}_{L^2(\D)}\norm{z}_{\HD{s}}+\norm{z-\tilde{z}_h}_{L^2(\D)}\norm{\tilde{u}_h}_{\HD{s}}\\
\lesssim\left(h^{\alpha+1/2-\varepsilon}+\eta\right)\left(\norm{f}_{H^{1/2-\varepsilon}}+\norm{\ud}_{H^{1/2-\varepsilon}}\right),
\end{multline*}
where $\alpha=\min\{s,1/2-\varepsilon\}$.
Eventually, we estimate the error due to the affine inerpolation of the bilinear form. In particular, invoking Lemma~\ref{lem:derivative_bilform_error} with $s_{2,k}=\bar{s}_2(s)$ and $s_{1,k}=s_{\min,k}$, where we recall $\overline{s}_2(s)=\min\{1,s_{\min,k}+1/2\}-\varepsilon$ for \(s\in\mathcal{S}_{k}\), and using the continuous embedding $\HD{s+1/2-\varepsilon}\hookrightarrow\HD{\overline{s}_2(s)}$, we deduce
\begin{align*}
\abs{a^\prime_s(\tilde{u}_h,\tilde{z}_h;\q)-\tilde{a}^\prime_s(\tilde{u}_h,\tilde{z}_h;\q)}\lesssim\tilde{\eta}\norm{\tilde{u}_h}_{\HD{s+1/2-\varepsilon}}\norm{\tilde{z}_h}_{\HD{s}}
\lesssim \tilde{\eta},
\end{align*}
where, as before, we used the fact that $\norm{\tilde{u}_h}_{\HD{s+1/2-\varepsilon}}\lesssim 1$, thanks to~\eqref{eq:state_norm_1}, and $\tilde{\eta}=\max_{k}\tilde{\eta}_k$. Then, combining all above estimates we obtain the desired result.
\end{proof}

\section{Numerical results}\label{sec:numerics}

In this section, we illustrate the theoretical results via numerical experiments.
We first report on the accuracy of the interpolant of the bilinear form \(a(\cdot,\cdot)\) and its \(s\)-derivative, and then solve identification problems either for \(s\) only or \(s\) and \(\delta\) jointly.

In all cases, a panel clustering approach is used to avoid prohibitively expensive dense matrices~\cite{AinsworthGlusa2018}.
In a nutshell, assembly is split into near field and far field interactions.
Far field interactions are then approximated using Chebyshev interpolation, where order and interpolation domain depend on the distance of the interaction.
Specially designed quadrature rules account for the singular behavior of the kernel, leading to assembly cost of \(\mathcal{O}(N \log^{2d} N)\) for fixed $\delta$ and $s$, where \(N=\dim V_{h}\) \cite{AinsworthGlusa2017}.
The error due to quadrature is always dominated by the discretization error, which is why we do not explicitly keep track of it and why we still denote the bilinear form involving quadrature as \(a(\cdot,\cdot;\cdot)\).
The equations for state and adjoint are solved using the conjugate gradient method, preconditioned by geometric multigrid, resulting in quasi-optimal complexity for both problem setup and solution.
The discretized control problems are solved using the BFGS algorithm~\cite{NocedalWright2006_NumericalOptimization2nd}.

We will consider the following test problems:
\begin{enumerate}
\item[\textbf{I}]
  The fractional Laplacian of order \(s\) on the \(n\)-dimensional ball.\\
  Let \(\Omega=B_{1}(0)\subset \mathbb{R}^{n}\), \(\delta=\infty\) and \((-\Delta)^{s}:=C_{n,s}\mathcal{L}\), where \(\mathcal{L}\) is given by \eqref{nonl_operator} and \(C_{n,s}\) is given by \eqref{eq:FL_constant}.
  The analytic solution of
  \begin{align*}
    \left\{\begin{array}{rcll}
             (-\Delta)^{s}u&=& 1 & \text{in } \Omega,\\
             u&=&0& \text{in } \mathbb{R}^{n}\setminus\Omega.
           \end{array}\right.
  \end{align*}
  is given by \cite{Getoor1961_FirstPassageTimesSymmetric}
  \begin{align*}
    u_{ex}(\x;s)=c_{n,s} (1-\abs{\x}^{2})_{+}^{s} \in H_{\Omega}^{s+1/2-\varepsilon}, \quad\forall \varepsilon>0,
  \end{align*}
  with \(c_{n,s}= \frac{\Gamma(n/2)}{2^{2s}\Gamma(\frac{n+2s}{2})\Gamma(1+s)}\).
\item[\textbf{II}]
  The nonlocal operator \(\mathcal{L}\) of order \(s\) and truncation length \(\delta\) on the \(n\)-dimensional ball.\\
  We solve \eqref{eq:nonl_strong} with right-hand side \(f=1\), for which no analytic solution is known.
\end{enumerate}

\subsection{Convergence of the interpolation}\label{sec:interpolation-convergence}
In this section we illustrate the convergence of the interpolation for both the operator and the gradient of the functional.

\subsubsection{Convergence of the operator interpolation}\label{sec:interpolation-operator-convergence}

In order to illustrate the results of Theorem~\ref{thm:solution_error_affine}, we solve problem \textbf{I} for \(s=\frac{1}{10},\frac{2}{10},\dots,\frac{9}{10}\) and dimension \(n=1\) for mesh sizes \(h=2^{-j}\), \(j=4,\dots,10\), using the operator interpolation \(\tilde{a}(\cdot,\cdot; s,\infty)\) and the exact (up to quadrature error) operator \(a(\cdot,\cdot; s,\infty)\).
The free parameter \(\xi\) is chosen via brute-force minimization with respect to the total number of interpolation nodes.
We note that this procedure is inexpensive, since only intervals \(\mathcal{S}_{k}\) and interpolation orders \(M_{k}\) are manipulated.

\begin{figure}
  \centering
  \includegraphics[width = 0.9\textwidth]{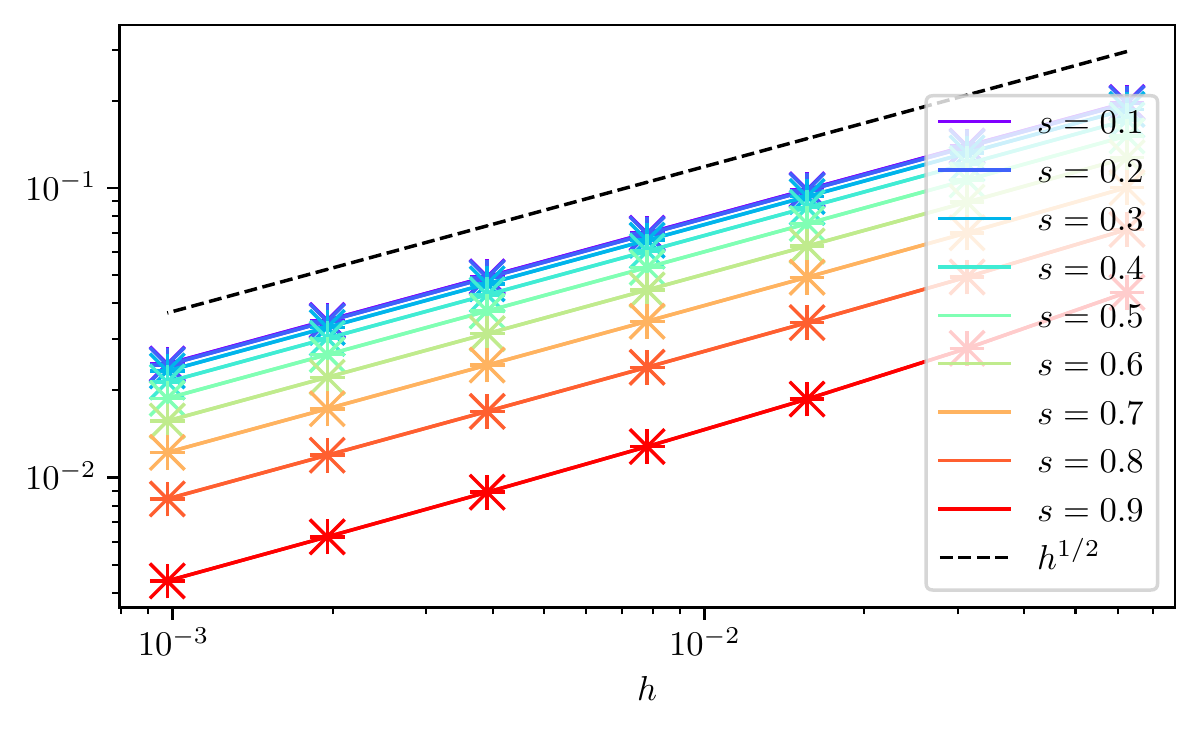}
  \includegraphics[width = 0.88\textwidth]{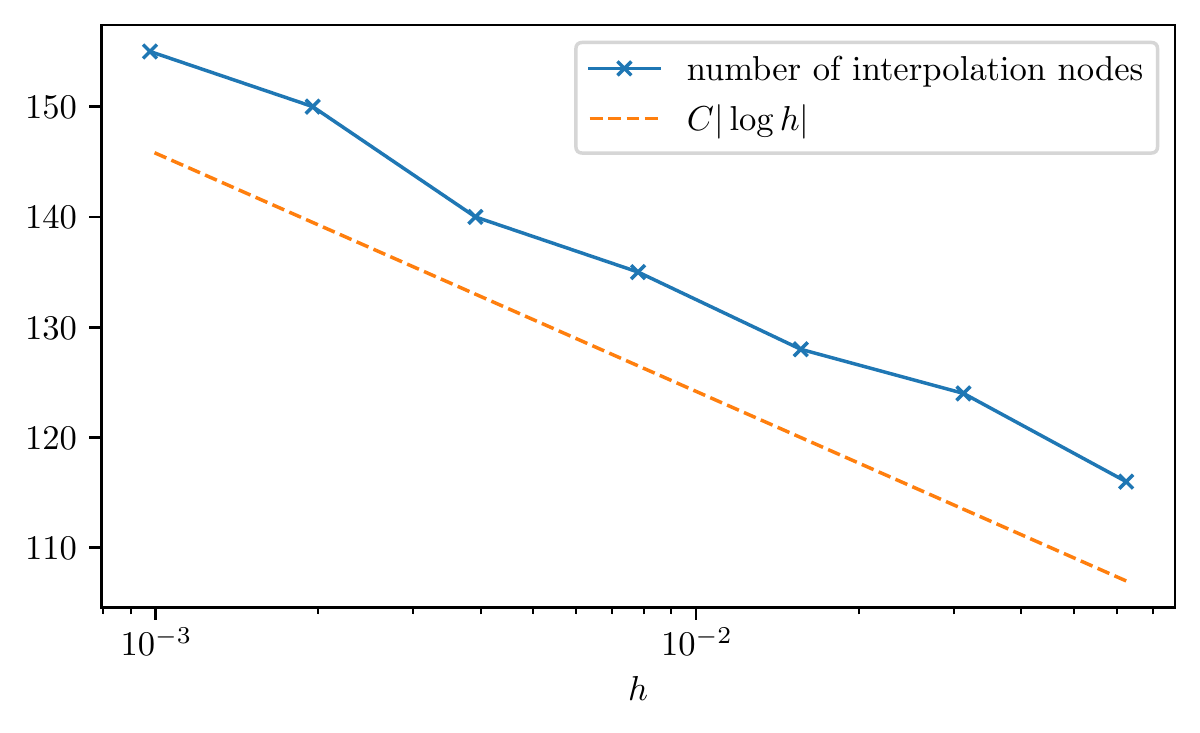}
  \includegraphics[width = 0.9\textwidth]{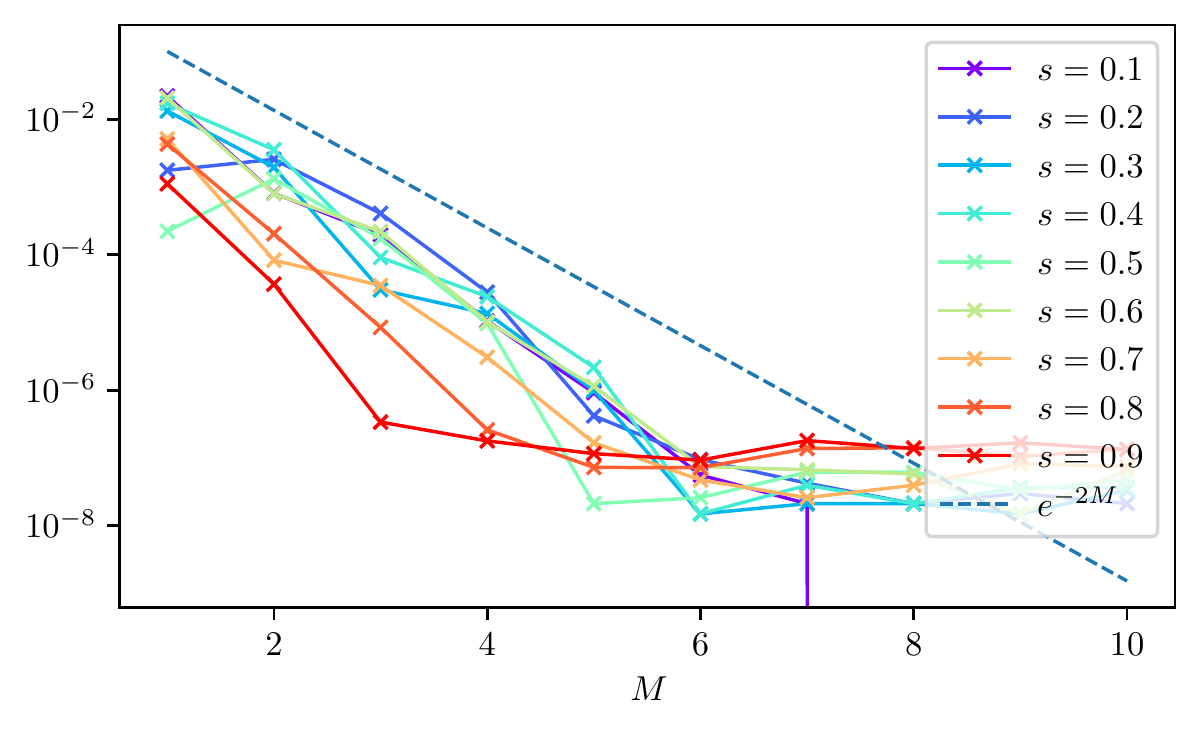}
  \caption{
    Solution errors with interpolation, \(\norm{\tilde{u}_{h}(s)-u(s)}_{H_{\Omega}^{s}}\) (\(\times\)), and without interpolation, \(\norm{u_{h}(s)-u(s)}_{H_{\Omega}^{s}}\) (\(+\)), for \(h=2^{-j}\), \(j=4,\dots,10\) (\emph{top}).
    Total number of interpolation nodes (\emph{middle}).
    Exponential convergence of the interpolation error \(\norm{\tilde{u}_{h,M}(s)-u_{h}(s)}_{H_{\Omega}^{s}}\)  with respect to \(M\) (\emph{bottom}).
  }
  \label{fig:interpolation-error}
\end{figure}

In Figure~\ref{fig:interpolation-error} we plot the errors \(\norm{\tilde{u}_{h}(s)-u(s)}_{H_{\Omega}^{s}}\) and \(\norm{u_{h}(s)-u(s)}_{H_{\Omega}^{s}}\), where \(\tilde{u}_{h}(s)\) and \(u_{h}(s)\) are the solutions obtained with and without interpolation.
It can be observed that both methods lead to virtually identical error.
We also plot the total number of interpolation nodes and observe that it indeed is proportional to \(\abs{\log h}\).
To further illustrate the exponential convergence with respect to the interpolation order, we also plot \(\norm{\tilde{u}_{h,M}(s)-u_{h}(s)}_{H_{\Omega}^{s}}\) for \(h=2^{-10}\), a fixed value of \(\xi\) and \(1\leq M\leq 12\) interpolation nodes on each interval \(\mathcal{S}_{k}\).
The observed plateau is due to quadrature error that impacts both \(u_{h,M}\) and \(u_{h}\).
However, we note that the magnitude of this error is negligible compared to the discretization error, as can be seen from Figure~\ref{fig:interpolation-error}.

\subsubsection{Convergence of the gradient approximation}\label{sec:gradient-convergence}

We evaluate the convergence of the approximation of the \(s\)-derivative.
We set \(u_{d}=1-\x^{2}\) and consider the optimal control problem with respect to $s$ for problem \textbf{I}.
For \(s=\frac{1}{10},\frac{2}{10},\dots,\frac{9}{10}\), we compute
\begin{align*}
  \abs{\tilde{j}_{h}'(s,\infty) - \tilde{j}_{\underline{h}}'(s,\infty)},
\end{align*}
where \(\tilde{j}_{h}\) is the reduced cost functional, evaluated on a mesh of size \(h\) (Figure~\ref{fig:derivative-interpolation-error}) and \(\underline{h}\ll h\).
We notice that the observed convergence rates are of order $\mathcal{O}(h)$, which indicate that the derived error estimates~\eqref{eq:error_obj_derivative} are nearly optimal for $s=1/2$.
We believe that an improved order of convergence could be obtained theoretically also for $s>1/2$, however this would require to re-derive the estimates~\eqref{eq:bdd_der_a_s} for $s_2,s_1\geq 1$, which is out of the scope of the current work.
We also report, for fixed mesh size \(h=2^{-10}\),
\begin{align*}
  \abs{\tilde{j}_{M}'(s,\infty) - \tilde{j}_{\overline{M}}'(s,\infty)},
\end{align*}
where \(\tilde{j}_{M}\) is the reduced cost functional using \(M\) interpolation nodes on each interval \(\mathcal{S}_{k}\) and \(M\ll \overline{M}\).
Again, we observe exponential convergence until the quadrature error is reached, which is in agreement with the theoretical results.

\begin{figure}
  \centering
  \includegraphics[width = 0.9\textwidth]{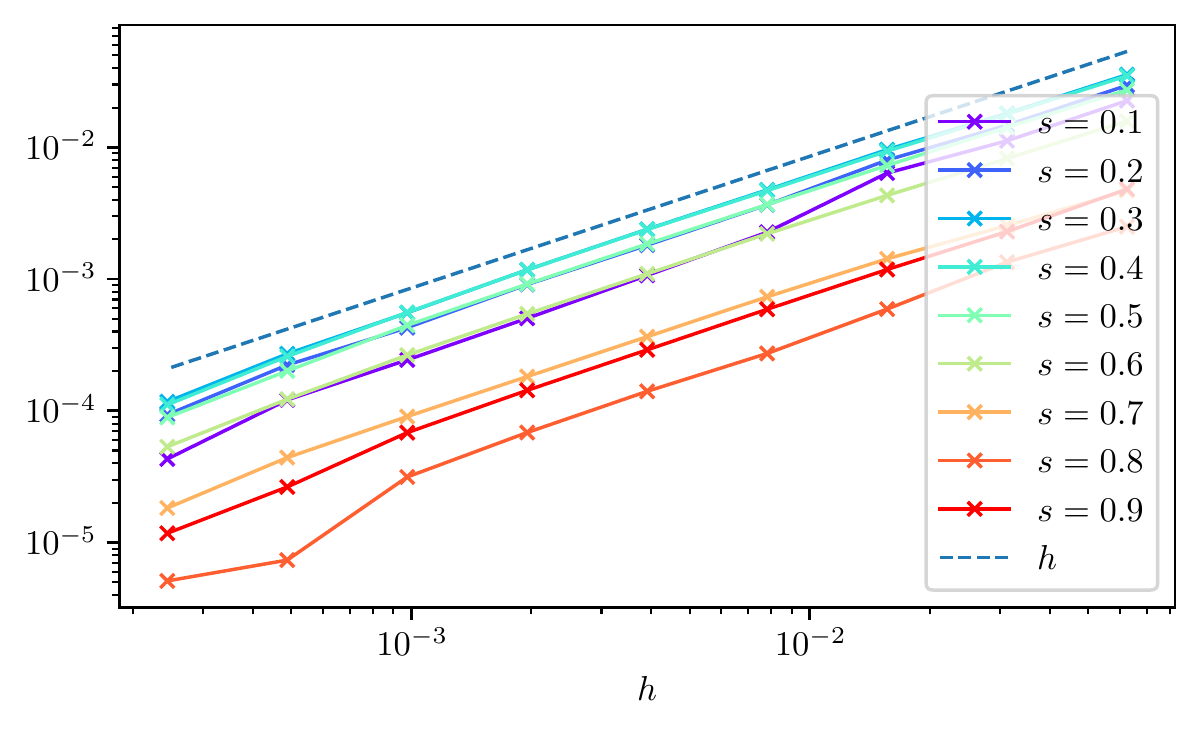}
  \includegraphics[width = 0.9\textwidth]{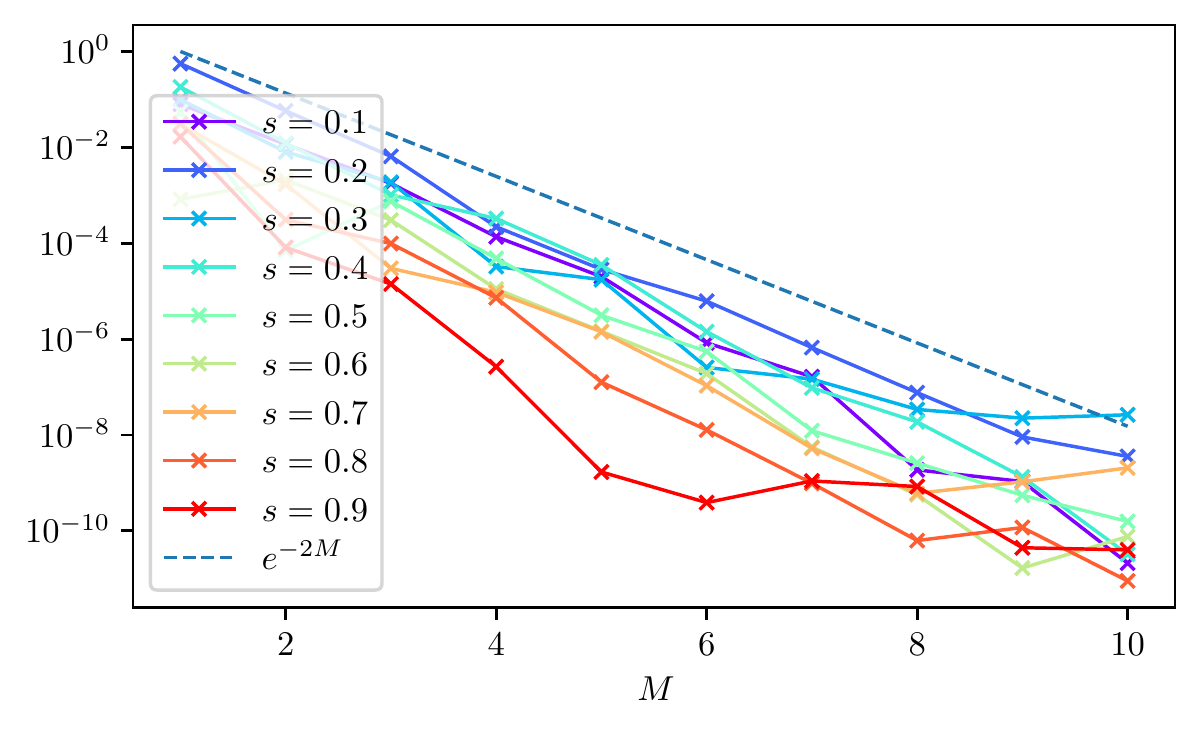}
  \caption{
    Convergence of the derivative of the reduced cost functional (\emph{top}).
    Exponential convergence of the derivative of the reduced cost functional with respect to \(M\) (\emph{bottom}).
  }
  \label{fig:derivative-interpolation-error}
\end{figure}

\subsection{Complexity and memory requirements}

For convenience, we summarize the complexity and memory requirements of the operator approximation for \(a(\cdot,\cdot;s,\delta)\).
In a preprocessing step, \(a(\cdot,\cdot;s_{m},\infty)\) is assembled for all Chebyshev nodes \(s_{m}\in\mathbb{S}_{k}\), \(k=1,\dots K\).
By combining Lemma~\ref{lem:global-approximation} with Theorem~\ref{thm:solution_error_affine}, the number of required Chebyshev nodes scales like \(\abs{\log h}\sim \log N\), where \(N=\dim V_{h}\).
By using the panel clustering approach of \cite{AinsworthGlusa2018}, each operator evaluation costs \(\mathcal{O}(N \log^{2n} N)\) operations and memory.
Hence, the overall cost of the preprocessing step scales as \(\mathcal O(N \log^{2n+1} N)\) in complexity and memory.

The assembly of the correction terms \(c(\cdot,\cdot;s_{m},\delta)\) for \(\delta<\infty\) costs \(\mathcal{O}(N \log^{2n} N)\) when using panel clustering.
Since matrix-vector products involving \(a\) and \(c\) scale in the same fashion as the assembly, solving linear systems involving \(a(\cdot,\cdot;s,\delta)\) for state and adjoint solution using multigrid is achieved in \(\mathcal{O}(N \log^{2n+1} N)\) operations.
Overall, each BFGS iteration has \(\mathcal{O}(N \log^{2n+1} N)\) complexity.
This shows that the presented approach is quasi-optimal in the number of unknowns.

We illustrate the efficiency of the approach by constructing all required operators of problem \textbf{II} for \(s\in\{0.25, 0.75\}\) and \(\delta\in\{0.5,1.5,2.5\}\) and different mesh sizes \(h\).
All computations were performed on a single core of a Intel Xeon E5-2650 processor.
\begin{figure}
  \centering
  \includegraphics[width = 0.9\textwidth]{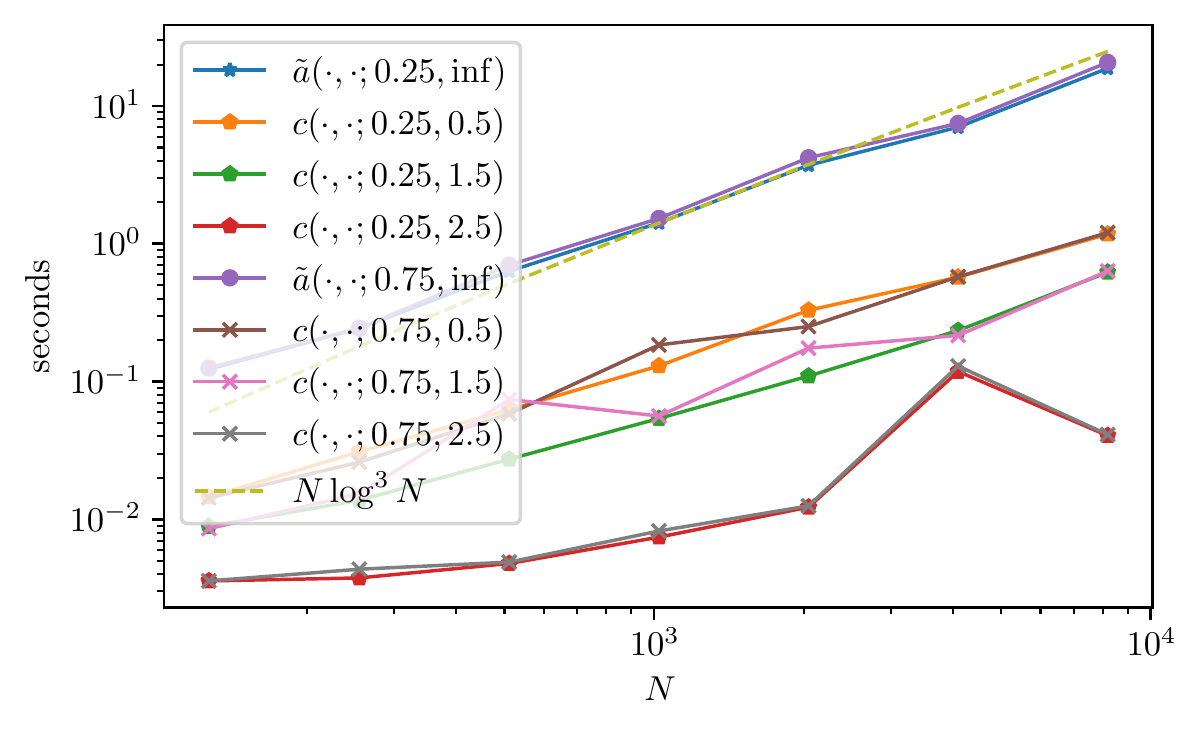}
  \includegraphics[width = 0.9\textwidth]{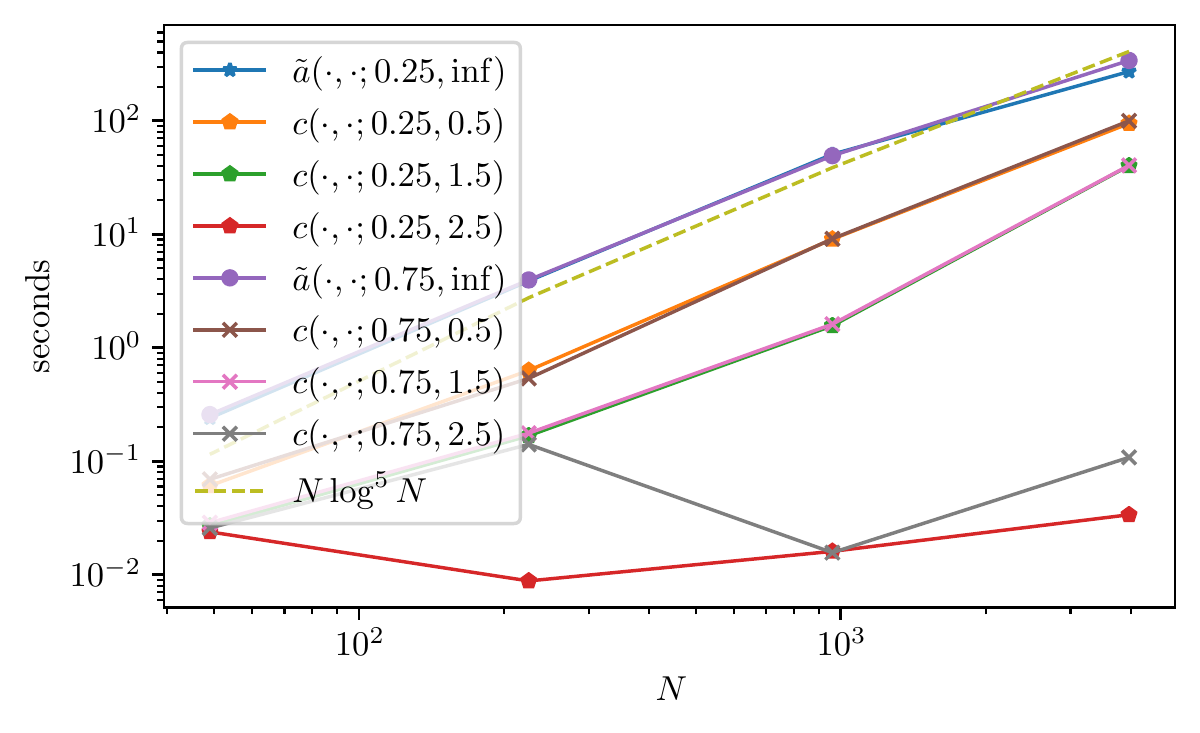}
  \caption{
    Assembly timings for \(\tilde{a}(\cdot,\cdot;s,\infty)\) and correction terms \(c(\cdot,\cdot;s,\delta)\) in 1D (\emph{top}) and 2D (\emph{bottom}).
  }
  \label{fig:timings}
\end{figure}
The timings for assembly of \(\tilde{a}(\cdot,\cdot;s,\infty)\) and \(c(\cdot,\cdot;s,\delta)\) given in Figure~\ref{fig:timings} show that we indeed recover quasi-optimal complexity.
We observe that assembly of the corrections for \(\delta\geq\operatorname{diam}\Omega\) is significantly cheaper, since the last term of \eqref{eq:correction} drops out, and assembly of the purely local mass terms only has \(\mathcal{O}(N)\) complexity.

\subsection{Consistency of the identification strategy}\label{sec:consistency}
For one- and two-dimensional problems we test the robustness of our optimization approach by using the manufactured solutions described in {\bf I} and {\bf II}.

\subsubsection{The fractional Laplacian case}

We consider test problem \textbf{I}, and set \(\mathcal{R}(s)=\frac{\alpha}{s(1-s)}\), \(\alpha=5\times 10^{-7}\) in order to satisfy conditions \eqref{eq:regularization-conditions}.
Set \(s^{*}=0.5\) and \(u_{d}=u_{ex}(s^{*})\).
The exact solution of the minimization problem \eqref{eq:pr_min} is given by \((s^{*}, u_{ex}( s^{*}))\), since \(s^{*}\) minimizes the regularization.
In \(n=1\) dimensions, the mesh size is given by \(h=2.4\times10^{-4}\) and in \(n=2\) dimensions by \(h=0.04\), and the number of unknowns is \(8191\) and \(3969\) respectively.
The initial guess is \(s_{0}=0.1\).
The linear solver tolerance is \(10^{-10}\), and the BFGS iteration is terminated if the norm of the gradient drops below \(10^{-8}\).
\begin{figure}
  \centering
  \includegraphics[width = 0.9\textwidth]{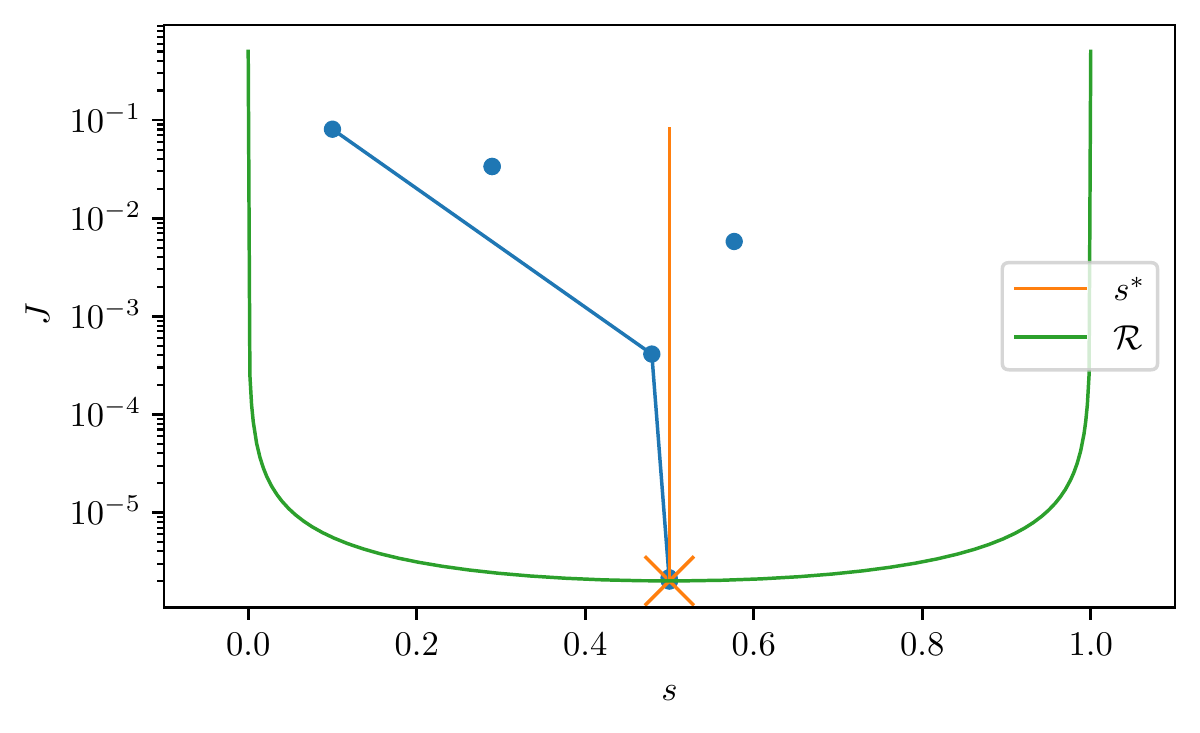}
  \includegraphics[width = 0.9\textwidth]{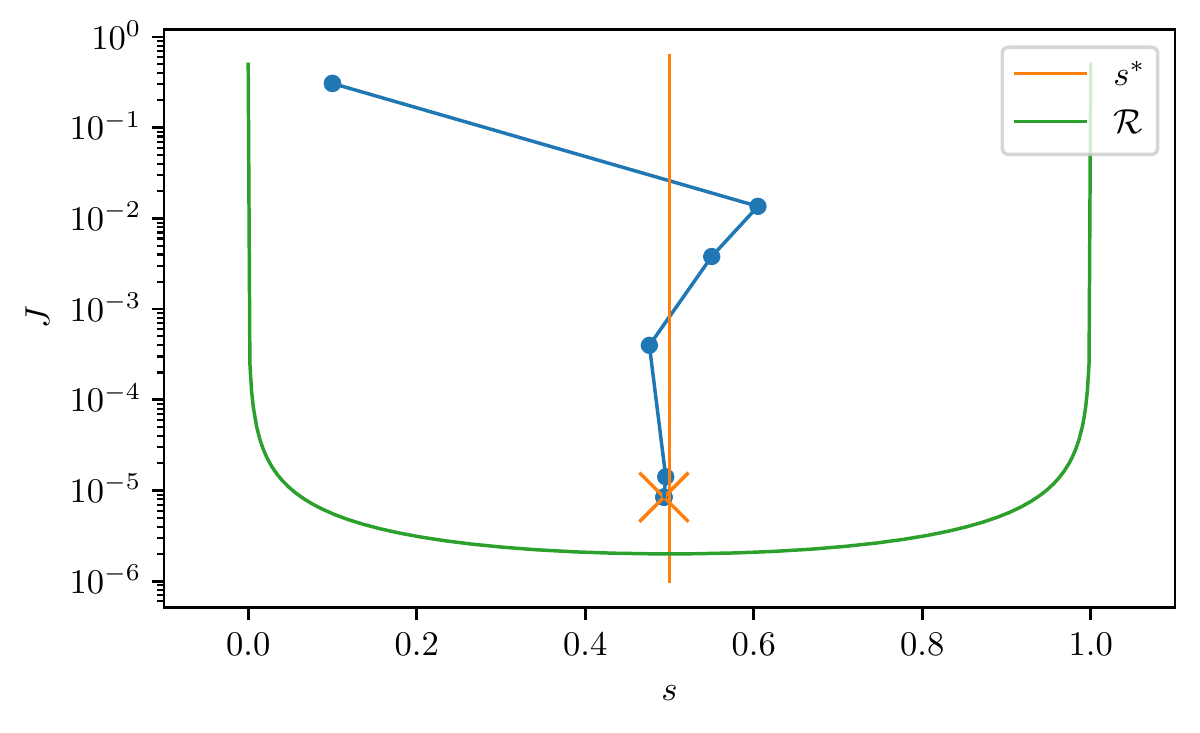}
  \caption{
    Value of the cost functional at all functional evaluations (\(\bullet\)) and for BFGS iterates (--) in 1D (\emph{top}) and 2D (\emph{bottom}) for test problem \textbf{I} and \(u_{d}=u_{ex}(0.5)\).
    The final iterate is marked with \(\times\).
  }
  \label{fig:iterates-fracLapl1}
\end{figure}
In Figure~\ref{fig:iterates-fracLapl1} we display the values \(\tilde{j}_{h}\) of the cost functional for all functional evaluations (\(\bullet\)).
In both cases, the exact value of \(s\) is recovered.

In a second example we take \(s^{*}=0.75\) and set \(u_{d}=u_{ex}(s^{*})\).
Since \(s^{*}\) no longer coincides with the minimum of the regularization, no exact solution of the minimization problem \eqref{eq:pr_min} is known.
The initial guess is again chosen to be \(s_{0}=0.1\).
\begin{figure}
  \centering
  \includegraphics[width = 0.9\textwidth]{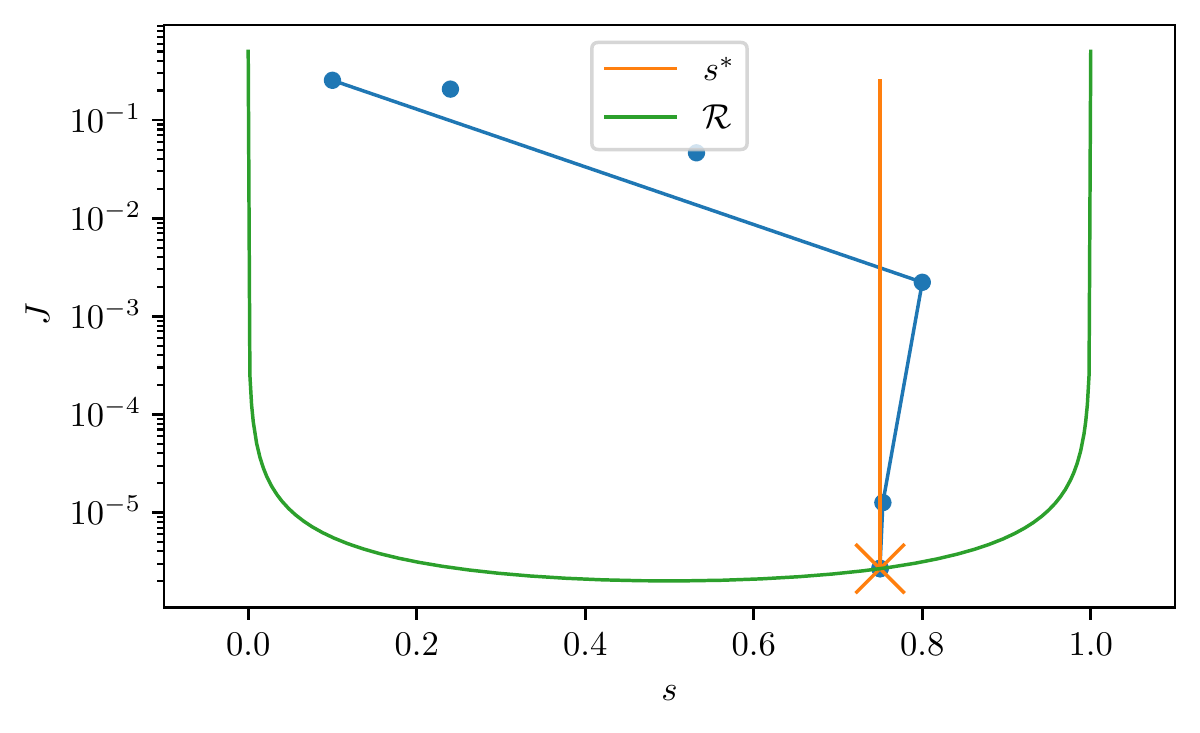}
  \includegraphics[width = 0.9\textwidth]{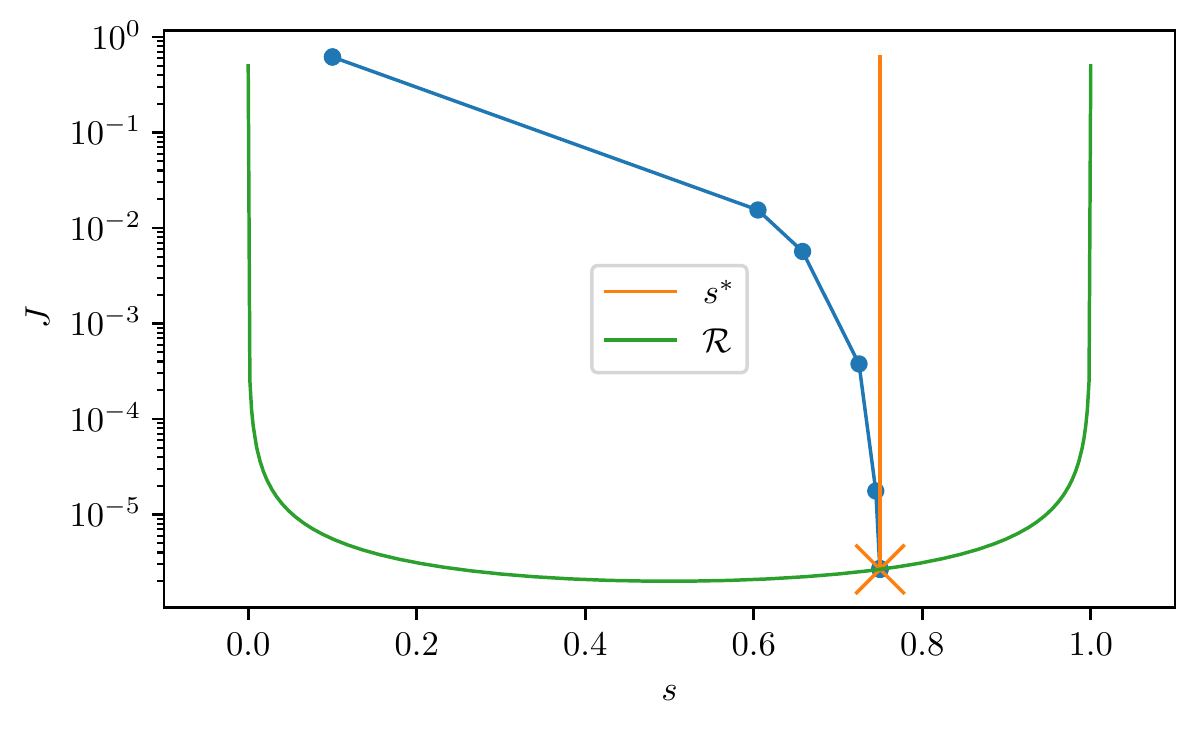}
  \caption{
    Value of the cost functional at all functional evaluations (\(\bullet\)) and for BFGS iterates (--) in 1D (\emph{top}) and 2D (\emph{bottom}) for test problem \textbf{I} and \(u_{d}=u_{ex}(0.75)\).
    The final iterate is marked with \(\times\).
  }
  \label{fig:iterates-fracLapl2}
\end{figure}
The values of the BFGS iterates and the regularization \(\mathcal{R}\) are shown in Figure~\ref{fig:iterates-fracLapl2}.
While the exact solution of the minimization problem \eqref{eq:pr_min} is unknown, it can be seen that in both cases \(s_{\text{final}}\), the final value of \(s\), is close to \(s^{*}\) and the final value of the cost functional is quite close to \(J(u_{d},s^{*})=\mathcal{R}(s^{*})\).

\subsubsection{The truncated fractional case}
\label{sec:trunc-fract-case}
In a final set of examples, we solve the identification problem for \(\q=(s,\delta)\).
We set \(\q^{*}=(s^{*},\delta^{*})=(0.75, 0.9)\) and solve the optimal control problem for \textbf{II} in \(n=1\) and \(n=2\) dimensions.
We set \(u_{d}=u_{h}(\q^{*})\) and use the regularization \(\mathcal{R}(s,\delta) = \frac{\alpha}{s(1-s)} + \beta\frac{e^{\delta}}{\delta}\) with \(\alpha=5\times 10^{-7}\) and \(\beta=10^{-6}\).
The initial guess is \(q_{0}=(0.1, 0.5)\).
\begin{figure}
  \centering
  \includegraphics[width = 0.9\textwidth]{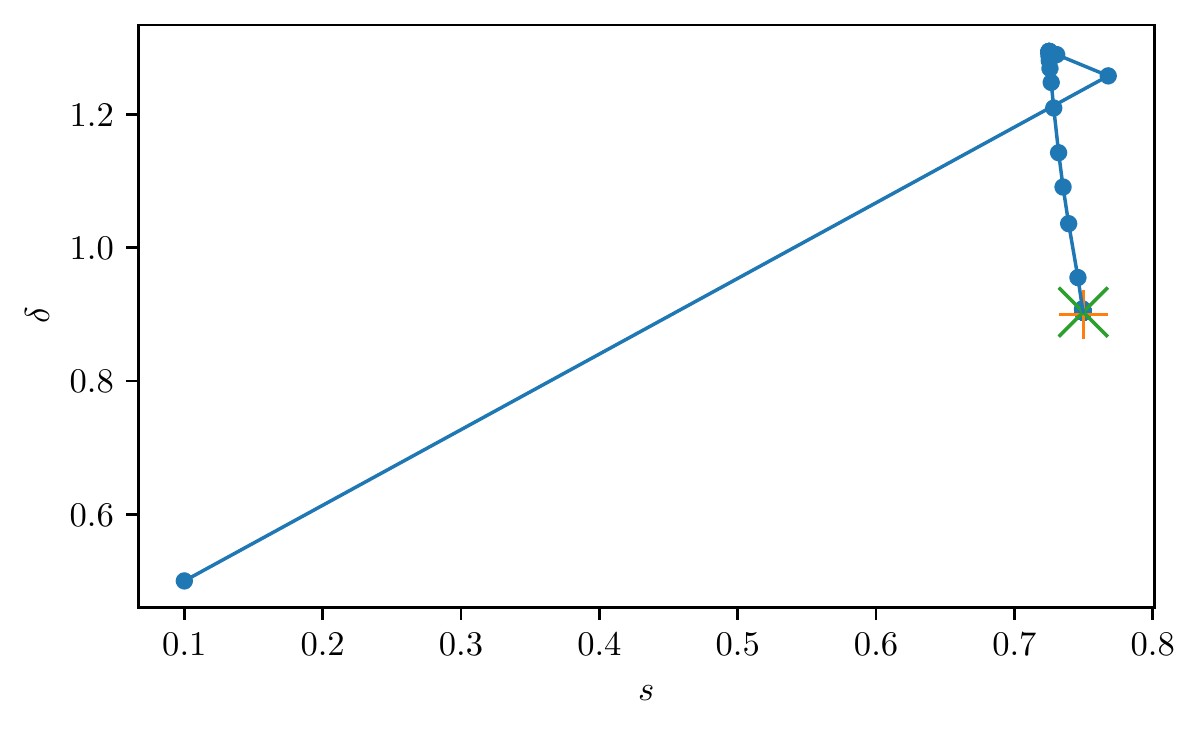}
  \includegraphics[width = 0.9\textwidth]{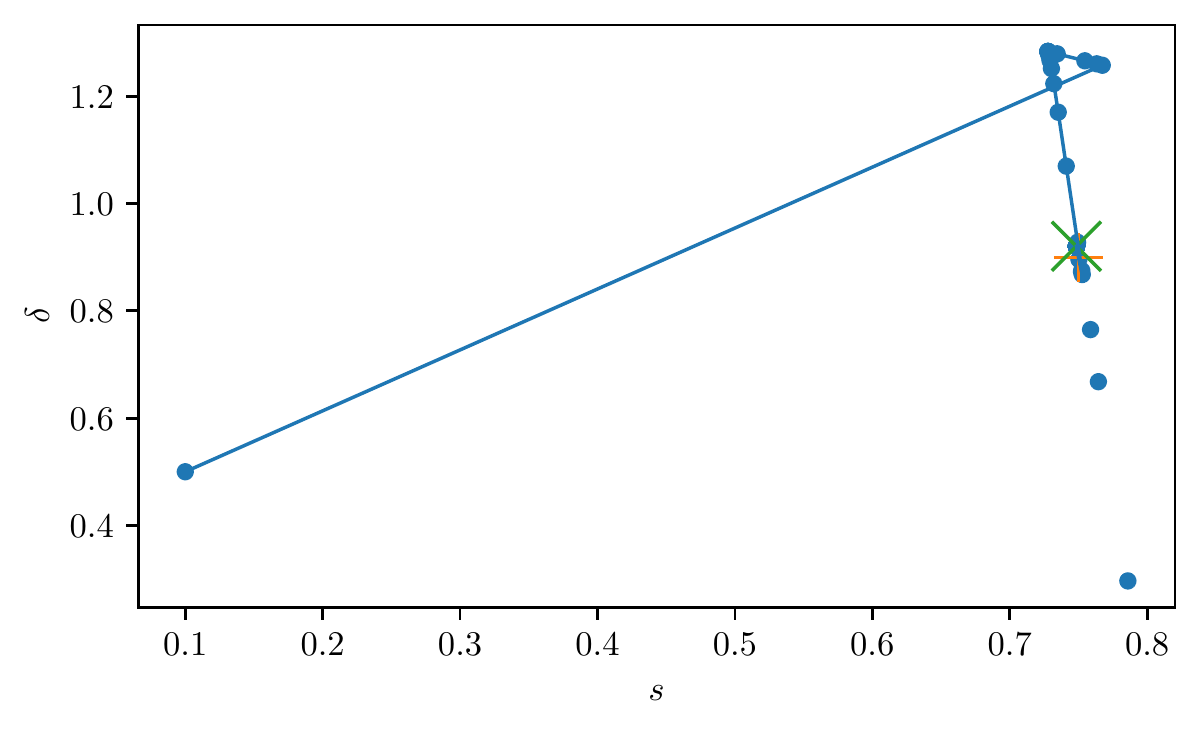}
  \caption{
    Cost functional evaluations (\(\bullet\)) and BFGS iterates (--) in 1D (\emph{top}) and 2D (\emph{bottom}) for test problem \textbf{II} and \(\mathcal{R}(s,\delta) = \frac{\alpha}{s(1-s)} + \beta\frac{e^{\delta}}{\delta}\) with \(\alpha=5\times 10^{-7}\) and \(\beta=10^{-6}\).
    The final iterate is marked with \(\times\), the parameters \(q^{*}\) used to generate the data \(u_{d}\) is marked with \(+\).
  }
  \label{fig:iterates-truncated-fracLapl}
\end{figure}
Figure~\ref{fig:iterates-truncated-fracLapl} shows the \((s,\delta)\)-values of the BFGS iterates.
We observe that the method indeed recovers \(q_{\text{final}}\) in proximity to \(q^{*}\).
\begin{figure}
  \centering
  \includegraphics[width = 0.9\textwidth]{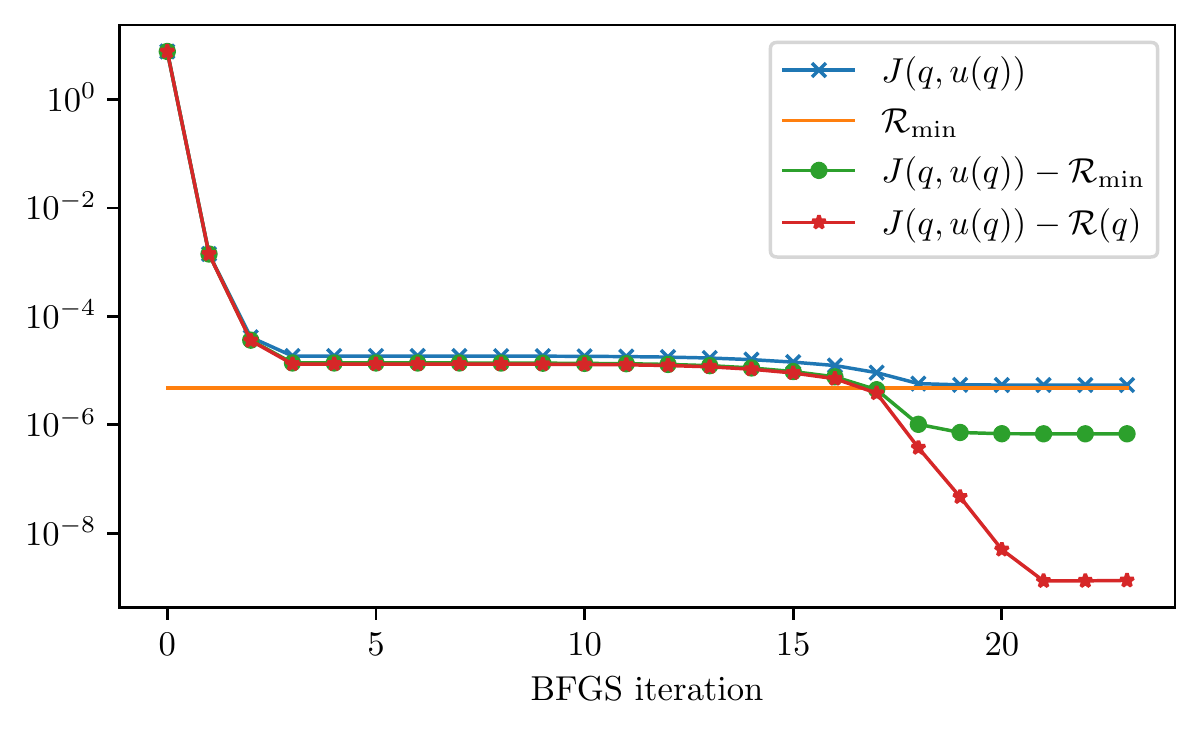}
  \includegraphics[width = 0.9\textwidth]{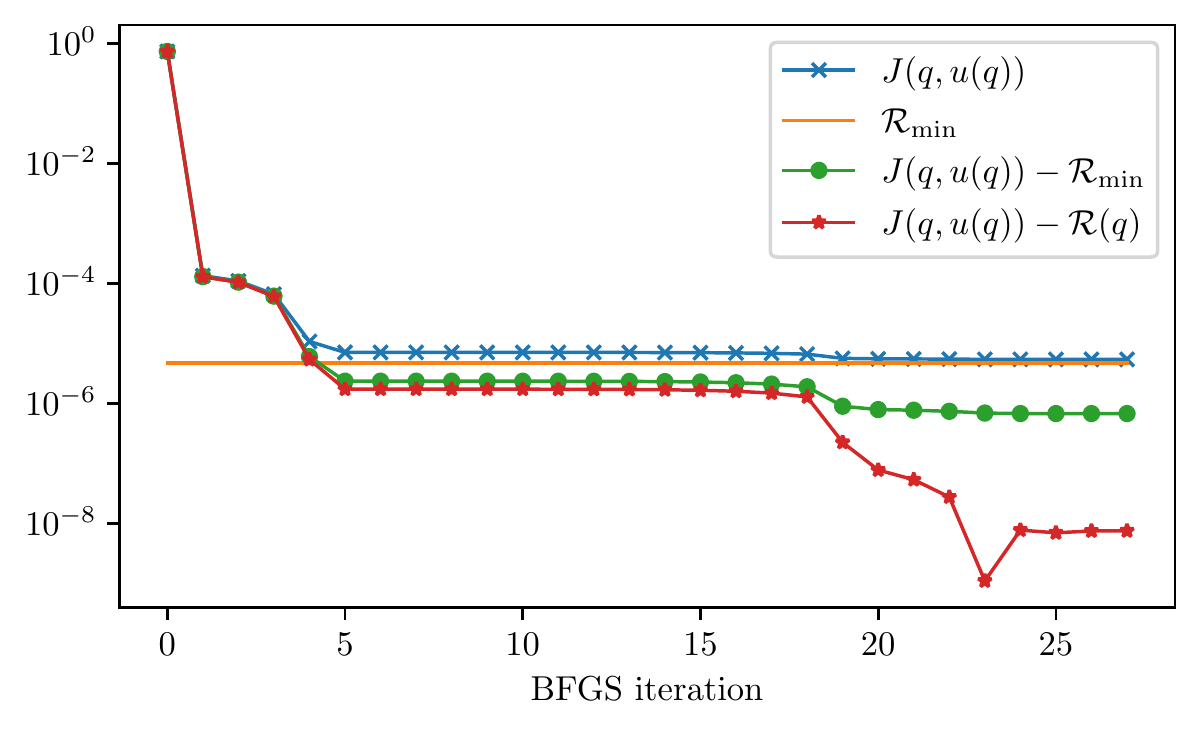}
  \caption{
    Values of the cost functional for the BFGS iterates in 1D (\emph{top}) and 2D (\emph{bottom}) for test problem \textbf{II} and \(\mathcal{R}(s,\delta) = \frac{\alpha}{s(1-s)} + \beta\frac{e^{\delta}}{\delta}\) with \(\alpha=5\times 10^{-7}\) and \(\beta=10^{-6}\).
    \(\mathcal{R}_{\min}=\min_{q}\mathcal{R}(q)\).
  }
  \label{fig:cost-truncated-fracLapl}
\end{figure}
In Figure~\ref{fig:cost-truncated-fracLapl}, we display the corresponding values of the cost functional.
We observe that \(j(q_{\text{final}})<j(q^{*})=\mathcal{R}(q^{*})\approx 5.4\times10^{-4}\).

\subsection{Convergence with respect to the mesh size}\label{sec:h-convergence}

In the same setting as in the previous section, we solve test problem \textbf{II} in \(n=1\) dimensions for a sequence of meshes with varying mesh size \(h\).
The linear solver tolerance is \(10^{-10}\), and the BFGS iteration is terminated if the norm of the gradient drops below \(10^{-8}\).
The obtained values for \(\q=(s,\delta)\), the number of BFGS iterations and the number of evaluations of the functional are shown in Table~\ref{tab:convergence}.
\begin{table}
  \centering
  \begin{tabular}{cc|cccc}
    \(h\) & \(N\) & \(s\) & \(\delta\) & iterations & evaluations \\
    \hline
    \(2^{-10}\) & 2047 & 0.74976 & 0.90306 & 24 & 77 \\
    \(2^{-11}\) & 4095 & 0.74976 & 0.90332 & 28 & 87 \\
    \(2^{-12}\) & 8191 & 0.74976 & 0.90329 & 23 & 79 \\
    \(2^{-13}\) & 16383 & 0.74975 & 0.90331 & 25 & 63 \\
    \(2^{-14}\) & 32767 & 0.74975 & 0.90331 & 23 & 73 \\
  \end{tabular}
  \caption{
    Convergence study for test problem \textbf{II} in 1D, \(\mathcal{R}(s,\delta) = \frac{\alpha}{s(1-s)} + \beta\frac{e^{\delta}}{\delta}\) with \(\alpha=5\times 10^{-7}\) and \(\beta=10^{-6}\).
    The optimal parameter values without regularization are \(s^{*}=0.75\) and \(\delta^{*}=0.9\).
  }
  \label{tab:convergence}
\end{table}
We observe that the number of required iterations barely changes as the mesh is refined.

\subsection{Convergence with respect to regularization}\label{sec:regularization}

In the same setting as in Section~\ref{sec:trunc-fract-case}, we solve test problem \textbf{II} in \(n=1\) dimensions for a sequence of regularization parameters \(\alpha\) and \(\beta\).
The obtained values for \(\q=(s,\delta)\), the number of BFGS iterations and the number of evaluations of the functional are shown in Table~\ref{tab:regularization}.
\begin{table}
  \centering
  \begin{tabular}{cc|cccc}
    \(\alpha\)& \(\beta\) & \(s\) & \(\delta\) & iterations & evaluations \\
    \hline
    5e-04 & 1e-03 & 0.73096 & 1.1373 & 24 & 67 \\
    5e-05 & 1e-04 & 0.7394 & 1.0478 & 21 & 55 \\
    5e-06 & 1e-05 & 0.7478 & 0.92978 & 25 & 74 \\
    5e-07 & 1e-06 & 0.74976 & 0.90329 & 23 & 79 \\
    5e-08 & 1e-07 & 0.74998 & 0.90034 & 25 & 29 \\
    5e-09 & 1e-08 & 0.75 & 0.90004 & 25 & 87 \\
  \end{tabular}
  \caption{
    Convergence with respect to the regularization for test problem \textbf{II} in 1D, \(\mathcal{R}(s,\delta) = \frac{\alpha}{s(1-s)} + \beta\frac{e^{\delta}}{\delta}\).
    The optimal parameter values without regularization are \(s^{*}=0.75\) and \(\delta^{*}=0.9\).
  }
  \label{tab:regularization}
\end{table}
We observe that as \(\alpha,\beta\rightarrow0\), the recovered values for \(s\) and \(\delta\) tend towards \((s^{*},\delta^{*})\).

\subsection{Noisy data}
\label{sec:noisy-data}

We again solve test problem \textbf{II} in 1D, but, as opposed to Section~\ref{sec:trunc-fract-case}, augment the data \(u_{d}\) by a pointwise normal distributed noise: \(u_{d}(\x)=u_{h}(\q^{*})(\x) + \mathcal{N}(0,\sigma^{2})\).
\begin{table}
  \centering
  \begin{tabular}{c|cccc}
    \(\sigma\)  & \(s\) & \(\delta\) & iterations & evaluations \\
    \hline
    1 & 0.8486 & 0.10847 & 35 & 105 \\
    \(2^{-1}\) & 0.79461 & 0.45007 & 32 & 136 \\
    \(2^{-2}\) & 0.76898 & 0.69291 & 28 & 67 \\
    \(2^{-3}\) & 0.75936 & 0.79458 & 27 & 69 \\
    \(2^{-4}\) & 0.75456 & 0.84792 & 23 & 71 \\
    \(2^{-5}\) & 0.75217 & 0.87517 & 20 & 69 \\
    \hline
    0 & 0.74976 & 0.90329 & 23 & 79 \\
  \end{tabular}
  \caption{
    Test problem \textbf{II} in 1D and \(\mathcal{R}(s,\delta) = \frac{\alpha}{s(1-s)} + \beta\frac{e^{\delta}}{\delta}\) with \(\alpha=5\times 10^{-7}\) and \(\beta=10^{-6}\) where the data \(u_{d}\) is pointwise disturbed by normally distributed noise with variance \(\sigma\).
    The optimal parameter values without regularization are \(s^{*}=0.75\) and \(\delta^{*}=0.9\).
  }
  \label{tab:noise}
\end{table}
We observe that as \(\sigma\rightarrow0\), both \(s\) and \(\delta\) converge towards their respective values obtained without noise.

\section{Conclusion}\label{sec:conclusion}
We presented a mathematically rigorous approach to parameter identification for nonlocal models featuring kernels of fractional type. Our careful analysis of well-posedness and regularity properties of state and adjoint equations allows for a deeper understanding of the identification problem and for the design of suitable optimization techniques. More specifically, by introducing an approximation, via interpolation, of the bilinear form and its derivative, we are able to obtain highly accurate approximations of the gradients with a nearly optimal complexity. We stress that the impact of this approximation goes beyond the scope of this paper and can potentially impact any numerical framework that involves parametrized nonlocal operators and their derivatives. As such, our approximation can be adapted to machine learning algorithms for the discovery of model parameters. 

A natural follow up of this work is to consider a higher-dimensional parameter space and compare our approach with alternative identification techniques. Nontrivial extensions of this work include the generalization to the variable horizon case and the variable fractional order case. The latter problem is particularly challenging, as regularity properties of the associated state equation have not been fully analyzed.

\bibliographystyle{siam}
\bibliography{references}
\end{document}